\documentclass{aspm}

\articleinfo{Version of May 6, 2021}{}{}

\usepackage{graphicx}
\usepackage{hyperref}
\usepackage[mathscr]{euscript}
\usepackage{amsmath}
\usepackage{amssymb}
\usepackage{amsthm}
\usepackage{bm}
\usepackage{enumitem}

\usepackage{caption}
\usepackage{subcaption}
\DeclareCaptionSubType{figure}

\allowdisplaybreaks[4]

\newtheorem{thm}{Theorem}[section]
\newtheorem{prop}[thm]{Proposition}
\newtheorem{lem}[thm]{Lemma}

\theoremstyle{definition}
\newtheorem{dfn}[thm]{Definition}

\theoremstyle{remark}
\newtheorem{rmk}[thm]{Remark}

\newtheorem*{ntn}{Notation}

\numberwithin{equation}{section}

\newcommand{\cent}{\mathop{\mathrm{cent}}\nolimits}
\newcommand{\curv}{\mathop{\mathrm{curv}}\nolimits}

\newcommand{\HM}{\mathscr{H}_{\Gamma}}
\newcommand{\id}[1]{\mathrm{id}_{#1}}
\newcommand{\im}{\mathop{\mathrm{Im}}\nolimits}
\newcommand{\ind}[1]{\mathbf{1}_{#1}}
\newcommand{\interior}{\mathop{\mathrm{int}}\nolimits}
\newcommand{\inv}{\mathop{\mathrm{Inv}}\nolimits}
\newcommand{\lip}{\mathop{\mathbf{Lip}}\nolimits}

\newcommand{\mob}{\mathop{\mbox{M\"{o}b}}\nolimits}
\newcommand{\rad}{\mathop{\mathrm{rad}}\nolimits}

\newcommand{\re}{\mathop{\mathrm{Re}}\nolimits}

\newcommand{\supp}{\mathop{\mathrm{supp}}\nolimits}
\newcommand{\thirdangle}{q}
\newcommand{\TDT}{\mathsf{TDT}^{+}}
\newcommand{\TDTo}{\mathsf{TDT}^{\oplus}}

\newcommand{\vol}{\mathop{\mathrm{vol}}\nolimits}

\title[The Laplacian on some self-conformal fractals and Weyl's asymptotics]{The Laplacian on some self-conformal fractals\\and Weyl's asymptotics for its eigenvalues:\\A survey of the analytic aspects}

\author[Naotaka Kajino]{Naotaka Kajino}\thanks{This work was supported by JSPS KAKENHI Grant Numbers JP25887038, JP15K17554, JP18K18720 and by the Research Institute for Mathematical Sciences, an International Joint Usage/Research Center located in Kyoto University.}

\address{Department of Mathematics, Graduate School of Science, Kobe University\\
Rokkodai-cho 1-1, Nada-ku, Kobe 657-8501, Japan}

\curraddr{Research Institute for Mathematical Sciences, Kyoto University\\
Kitashirakawa-Oiwake-cho, Sakyo-ku, Kyoto 606-8502, Japan}

\email{nkajino@kurims.kyoto-u.ac.jp}

\rcvdate{}
\rvsdate{}



\begin{document}

\begin{abstract}
This article surveys the analytic aspects of the author's recent studies on the
construction and analysis of a \emph{``geometrically canonical'' Laplacian} on circle
packing fractals invariant with respect to certain Kleinian groups
(i.e., discrete groups of M\"{o}bius transformations on the Riemann sphere
$\widehat{\mathbb{C}}=\mathbb{C}\cup\{\infty\}$), including the classical
\emph{Apollonian gasket} and some \emph{round Sierpi\'{n}ski carpets}.
The main result on Weyl's asymptotics for its eigenvalues is of the same form as
that by Oh and Shah [\emph{Invent.\ Math.}\ \textbf{187} (2012), 1--35, Theorem 1.4]
on the asymptotic distribution of the circles in a very large class of such fractals.
\end{abstract}

\maketitle
\section{Introduction}\label{sec:intro}
This article, which is a considerable expansion of \cite{K:WeylRSCsurv}, concerns the author's recent studies in
\cite{K:MFO2016,K:WeylAG,K:WeylDCGMaskit,K:WeylRSC} on Weyl's eigenvalue asymptotics
for a ``geometrically canonical'' Laplacian defined by the author on circle packing
fractals which are invariant with respect to certain Kleinian groups (i.e., discrete
groups of M\"{o}bius transformations on $\widehat{\mathbb{C}}:=\mathbb{C}\cup\{\infty\}$),
including the classical \emph{Apollonian gasket} (Figure \ref{fig:AGs}) and
some \emph{round Sierpi\'{n}ski carpets} (Figure \ref{fig:RSCs}).
Here we focus on sketching the construction of the Laplacian, the proof of its
uniqueness and basic properties, and the analytic aspects of the proof of the eigenvalue
asymptotics; the reader is referred to \cite{K:WeylSurvErgodicTh} for a survey
of the ergodic-theoretic aspects of the proof of the eigenvalue asymptotics.

This article is organized as follows. First in \S\ref{sec:AG-geometry} we introduce
the Apollonian gasket $K(\mathscr{D})$ and recall its basic geometric properties.
In \S\ref{sec:AG-DF}, after a brief summary of how the Laplacian on $K(\mathscr{D})$
was discovered by Teplyaev in \cite{Tep:energySG}, we give its definition and sketch
the proof of the result in \cite{K:WeylAG} that it is the infinitesimal generator of
the \emph{unique} strongly local, regular symmetric Dirichlet form over $K(\mathscr{D})$
with respect to which the inclusion map $K(\mathscr{D})\hookrightarrow\mathbb{C}$ is
\emph{harmonic} on the complement of the three outmost vertices. In \S\ref{sec:Weyl-AG},
we state the principal result in \cite{K:WeylAG} that the Laplacian on $K(\mathscr{D})$
satisfies Weyl's eigenvalue asymptotics of the same form as the asymptotic distribution of
the circles in $K(\mathscr{D})$ by Oh and Shah in \cite[Corollary 1.8]{OhShah:InventMath2012},
and sketch the proof of certain estimates on the eigenvalues required to conclude Weyl's
asymptotics by applying the ergodic-theoretic result explained in \cite{K:WeylSurvErgodicTh}.
Finally, in \S\ref{sec:WeylRSC} we present a partial extension of these results to the case of
round Sierpi\'{n}ski carpets which are invariant with respect to certain concrete Kleinian groups.
\begin{ntn}
We use the following notation throughout this article.
\begin{itemize}[label=\textup{(0)},align=left,leftmargin=*,topsep=0pt,parsep=0pt,itemsep=0pt]
\item[\textup{(0)}]The symbols $\subset$ and $\supset$ for set inclusion
	\emph{allow} the case of the equality.
\item[\textup{(1)}]$\mathbb{N}:=\{n\in\mathbb{Z}\mid n>0\}$, i.e., $0\not\in\mathbb{N}$.
\item[\textup{(2)}]$\widehat{\mathbb{C}}:=\mathbb{C}\cup\{\infty\}$ denotes the Riemann sphere.
\item[\textup{(3)}]$i:=\sqrt{-1}$ denotes the imaginary unit. The real and imaginary
	parts of $z\in\mathbb{C}$ are denoted by $\re z$ and $\im z$, respectively.
\item[\textup{(4)}]The cardinality (number of elements) of a set $A$ is denoted by $\#A$.
\item[\textup{(5)}]Let $E$ be a non-empty set. We define $\id{E}:E\to E$ by $\id{E}(x):=x$.
	For $x\in E$, we define $\ind{x}=\ind{x}^{E}\in\mathbb{R}^{E}$ by
	$\ind{x}(y):=\ind{x}^{E}(y):=\bigl\{\begin{smallmatrix}1 & \textrm{if $y=x$,}\\ 0 & \textrm{if $y\not=x$.}\end{smallmatrix}$
	For $u:E\to[-\infty,+\infty]$ we set $\|u\|_{\sup}:=\|u\|_{\sup,E}:=\sup_{x\in E}|u(x)|$.
\item[\textup{(6)}]Let $E$ be a topological space. The Borel $\sigma$-field of $E$ is
	denoted by $\mathscr{B}(E)$. For $A\subset E$, its interior, closure and boundary
	in $E$ are denoted by $\interior_{E}A$, $\overline{A}^{E}$ and $\partial_{E}A$,
	respectively, and when $E=\mathbb{C}$ they are simply denoted by $\interior A$,
	$\overline{A}$ and $\partial A$, respectively. We set
	$\mathcal{C}(E):=\{u\mid\textrm{$u:E\to\mathbb{R}$, $u$ is continuous}\}$,
	$\supp_{E}[u]:=\overline{u^{-1}(\mathbb{R}\setminus\{0\})}^{E}$ for $u\in\mathcal{C}(E)$, and
	$\mathcal{C}_{\mathrm{c}}(E):=\{u\in\mathcal{C}(E)\mid\textrm{$\supp_{E}[u]$ is compact}\}$.
\item[\textup{(7)}]Let $n\in\mathbb{N}$. The Lebesgue measure on
	$(\mathbb{R}^{n},\mathscr{B}(\mathbb{R}^{n}))$ is denoted by $\vol_{n}$.
	The Euclidean inner product and norm on $\mathbb{R}^{n}$ are denoted by
	$\langle\cdot,\cdot\rangle$ and $|\cdot|$, respectively.
	For $A\subset\mathbb{R}^{n}$ and $f:A\to\mathbb{C}$ we set 
	$\lip_{A}f:=\sup_{x,y\in A,\,x\not=y}\frac{|f(x)-f(y)|}{|x-y|}$
	($\sup\emptyset:=0$). For a non-empty open subset $U$ of $\mathbb{R}^{n}$ and
	$u:U\to\mathbb{R}$ with $\lip_{U}u<+\infty$, the first-order partial derivatives
	of $u$, which exist $\vol_{n}$-a.e.\ on $U$, are denoted by
	$\partial_{1}u,\ldots,\partial_{n}u$, and we set
	$\nabla u:=(\partial_{1}u,\ldots,\partial_{n}u)$.
\end{itemize}
\end{ntn}
\section{The Apollonian gasket and its fractal geometry}\label{sec:AG-geometry}
In this section, we introduce the Apollonian gasket and state its geometric properties
needed for our purpose. The same framework is presented also in \cite[Section 2]{K:WeylSurvErgodicTh},
but we repeat it here for the reader's convenience. The following definition and proposition form
the basis of the construction and further detailed studies of the Apollonian gasket.
\begin{dfn}[tangential disk triple]\label{dfn:tangential-disk-triple}
\begin{itemize}[label=\textup{(0)},align=left,leftmargin=*,topsep=0pt,parsep=0pt,itemsep=0pt]
\item[\textup{(0)}]We set $S:=\{1,2,3\}$.
\item[\textup{(1)}]Let $D_{1},D_{2},D_{3}\subset\mathbb{C}$ be either three open disks
	or two open disks and an open half-plane. The triple $\mathscr{D}:=(D_{1},D_{2},D_{3})$
	of such sets is called a \emph{tangential disk triple}
	if and only if $\#(\overline{D_{j}}\cap\overline{D_{k}})=1$
	(i.e., $D_{j}$ and $D_{k}$ are \emph{externally} tangent) for any $j,k\in S$
	with $j\not=k$.
	If $\mathscr{D}$ is such a triple consisting of three disks, then the open triangle in $\mathbb{C}$
	with vertices the centers of $D_{1},D_{2},D_{3}$ is denoted by $\triangle(\mathscr{D})$.
\item[\textup{(2)}]Let $\mathscr{D}=(D_{1},D_{2},D_{3})$ be a tangential disk triple.
	The open subset $\mathbb{C}\setminus\bigcup_{j\in S}\overline{D_{j}}$ of
	$\mathbb{C}$ is then easily seen to have a unique bounded connected component,
	which is denoted by $T(\mathscr{D})$ and called the \emph{ideal triangle}
	associated with $\mathscr{D}$. We also set
	$\{q_{j}(\mathscr{D})\}:=\overline{D_{k}}\cap\overline{D_{l}}$
	for each $(j,k,l)\in\{(1,2,3),(2,3,1),(3,1,2)\}$ and
	$V_{0}(\mathscr{D}):=\{q_{j}(\mathscr{D})\mid j\in S\}$.
\item[\textup{(3)}]A tangential disk triple $\mathscr{D}=(D_{1},D_{2},D_{3})$ is
	called \emph{positively oriented} if and only if its associated ideal triangle
	$T(\mathscr{D})$ is to the left of $\partial T(\mathscr{D})$
	when $\partial T(\mathscr{D})$ is oriented so as to have
	$\{q_{j}(\mathscr{D})\}_{j=1}^{3}$ in this order.
\end{itemize}
\noindent
Finally, we define
\begin{align*}
\TDT&:=\{\mathscr{D}\mid\textrm{$\mathscr{D}$ is a positively oriented tangential disk triple}\},\\
\TDTo&:=\{\mathscr{D}\mid\textrm{$\mathscr{D}=(D_{1},D_{2},D_{3})\in\TDT$, $D_{1},D_{2},D_{3}$ are disks}\}.
\end{align*}
\end{dfn}
The following proposition is classical and can be shown
by some elementary (though lengthy) Euclidean-geometric arguments.
We set $\rad(D):=r$ and $\curv(D):=r^{-1}$ for each open disk $D\subset\mathbb{C}$
of radius $r\in(0,+\infty)$ and $\curv(D):=0$ for each open half-plane $D\subset\mathbb{C}$.
\begin{prop}\label{prop:circumscribed-inscribed}
Let $\mathscr{D}=(D_{1},D_{2},D_{3})\in\TDT$, set
$(\alpha,\beta,\gamma):=\bigl(\curv(D_{1}),\curv(D_{2}),\curv(D_{3})\bigr)$
and set $\kappa:=\kappa(\mathscr{D}):=\sqrt{\beta\gamma+\gamma\alpha+\alpha\beta}$.%
\begin{itemize}[label=\textup{(1)},align=left,leftmargin=*,topsep=0pt,parsep=0pt,itemsep=0pt]
\item[\textup{(1)}]Let $D_{\mathrm{cir}}(\mathscr{D})\subset\mathbb{C}$ denote
	the \emph{circumscribed disk} of $T(\mathscr{D})$, i.e., the unique open disk with
	$\{q_{1}(\mathscr{D}),q_{2}(\mathscr{D}),q_{3}(\mathscr{D})\}
		\subset\partial D_{\mathrm{cir}}(\mathscr{D})$.
	Then
	$\overline{T(\mathscr{D})}\setminus\{q_{1}(\mathscr{D}),q_{2}(\mathscr{D}),q_{3}(\mathscr{D})\}
		\subset D_{\mathrm{cir}}(\mathscr{D})$,
	$\partial D_{\mathrm{cir}}(\mathscr{D})$ is orthogonal to $\partial D_{j}$
	for any $j\in S$, and $\curv(D_{\mathrm{cir}}(\mathscr{D}))=\kappa$.
\item[\textup{(2)}]There exists a unique \emph{inscribed disk}
	$D_{\mathrm{in}}(\mathscr{D})$ of $T(\mathscr{D})$,
	i.e., a unique open disk $D_{\mathrm{in}}(\mathscr{D})\subset\mathbb{C}$
	such that $D_{\mathrm{in}}(\mathscr{D})\subset T(\mathscr{D})$ and
	$\#(\overline{D_{\mathrm{in}}(\mathscr{D})}\cap\overline{D_{j}})=1$
	for any $j\in S$. Moreover,
	$\curv(D_{\mathrm{in}}(\mathscr{D}))=\alpha+\beta+\gamma+2\kappa$.
\end{itemize}
\end{prop}
The following notation is standard in studying self-similar sets.
\begin{dfn}\label{dfn:word-shift}
\begin{itemize}[label=\textup{(1)},align=left,leftmargin=*,topsep=0pt,parsep=0pt,itemsep=0pt]
\item[\textup{(1)}]We set $W_{0}:=\{\emptyset\}$,
	where $\emptyset$ is an element called the \emph{empty word}, $W_{m}:=S^{m}$
	for $m\in\mathbb{N}$ and $W_{*}:=\bigcup_{m\in\mathbb{N}\cup\{0\}}W_{m}$.
	For $w\in W_{*}$, the unique $m\in\mathbb{N}\cup\{0\}$ satisfying $w\in W_{m}$
	is denoted by $|w|$ and called the \emph{length} of $w$.
\item[\textup{(2)}]Let $w,v\in W_{*}$, $w=w_{1}\ldots w_{m}$, $v=v_{1}\ldots v_{n}$.
	We define $wv\in W_{*}$ by $wv:=w_{1}\ldots w_{m}v_{1}\ldots v_{n}$
	($w\emptyset:=w$, $\emptyset v:=v$). We also define
	$w^{(1)}\ldots w^{(k)}$ for $k\geq 3$ and $w^{(1)},\ldots,w^{(k)}\in W_{*}$
	inductively by $w^{(1)}\ldots w^{(k)}:=(w^{(1)}\ldots w^{(k-1)})w^{(k)}$.
	For $w\in W_{*}$ and $n\in\mathbb{N}\cup\{0\}$ we set $w^{n}:=w\ldots w\in W_{n|w|}$.
	We write $w\leq v$ if and only if $w=v\tau$ for some $\tau\in W_{*}$,
	and write $w\not\asymp v$ if and only if neither $w\leq v$ nor $v\leq w$ holds.
\end{itemize}
\end{dfn}
Proposition \ref{prop:circumscribed-inscribed}-(2) enables us to define natural
``contraction maps'' $\Phi_{w}:\TDT\to\TDT$ for each $w\in W_{*}$, which in turn
is used to define the Apollonian gasket $K(\mathscr{D})$ associated with
$\mathscr{D}\in\TDT$, as follows.
\begin{dfn}\label{dfn:TDT-Phi}
We define maps $\Phi_{1},\Phi_{2},\Phi_{3}:\TDT\to\TDT$ by
\begin{equation}\label{eq:TDT-Phi}
\begin{cases}
\Phi_{1}(\mathscr{D}):=(D_{\mathrm{in}}(\mathscr{D}),D_{2},D_{3}),\\
\Phi_{2}(\mathscr{D}):=(D_{1},D_{\mathrm{in}}(\mathscr{D}),D_{3}),\quad\mathscr{D}=(D_{1},D_{2},D_{3})\in\TDT.\\
\Phi_{3}(\mathscr{D}):=(D_{1},D_{2},D_{\mathrm{in}}(\mathscr{D})),
\end{cases}
\end{equation}
We also set $\Phi_{w}:=\Phi_{w_{m}}\circ\cdots\circ\Phi_{w_{1}}$
($\Phi_{\emptyset}:=\id{\TDT}$) and $\mathscr{D}_{w}:=\Phi_{w}(\mathscr{D})$
for $w=w_{1}\ldots w_{m}\in W_{*}$ and $\mathscr{D}\in\TDT$.
\end{dfn}
%
\begin{figure}[t]\captionsetup{width=\linewidth}\centering
\subcaptionbox{Examples without a half-plane\label{fig:AGs-disk}}[.520\linewidth]{%
	\includegraphics[height=70pt]{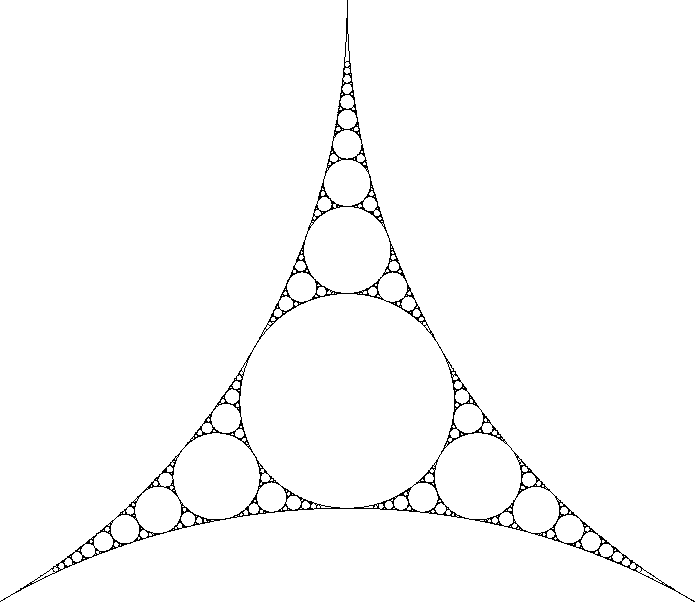}\hspace*{5pt}\includegraphics[height=70pt]{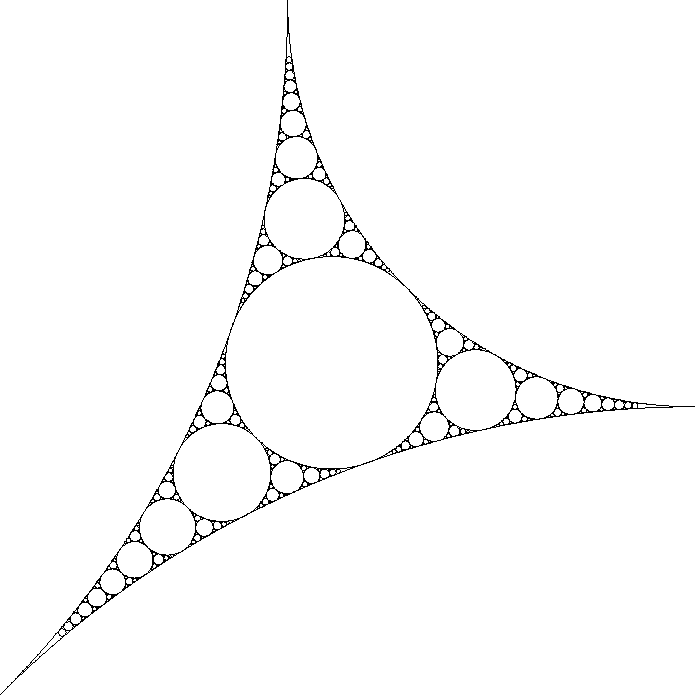}}
\subcaptionbox{Example with a half-plane\label{fig:AG-halfplane}}[.460\linewidth]{%
	\includegraphics[height=70pt]{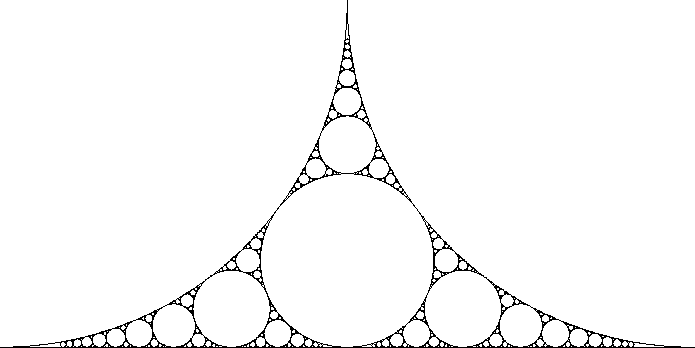}}
\caption{The Apollonian gaskets $K(\mathscr{D})$ associated with $\mathscr{D}\in\TDT$}
\label{fig:AGs}
\end{figure}
%
\begin{dfn}[Apollonian gasket]\label{dfn:AG-construction}
Let $\mathscr{D}\in\TDT$. We define the \emph{Apollonian gasket}
$K(\mathscr{D})$ associated with $\mathscr{D}$ (see Figure \ref{fig:AGs}) by
\begin{equation}\label{eq:AG-construction}
K(\mathscr{D}):=\overline{T(\mathscr{D})}\setminus\bigcup\nolimits_{w\in W_{*}}D_{\mathrm{in}}(\mathscr{D}_{w})
	=\bigcap\nolimits_{m\in\mathbb{N}}\bigcup\nolimits_{w\in W_{m}}\overline{T(\mathscr{D}_{w})}.
\end{equation}
\end{dfn}
The curvatures of the disks involved in \eqref{eq:AG-construction}
admit the following simple expression.
\begin{dfn}\label{dfn:Mw}
We define $4\times 4$ real matrices $M_{1},M_{2},M_{3}$ by
\begin{equation}\label{eq:M1M2M3}\setlength\arraycolsep{3pt}
M_{1}:=\begin{pmatrix}1 & 0 & 0 & 0\\1 & 1 & 0 & 1\\1 & 0 & 1 & 1\\2 & 0 & 0 & 1\end{pmatrix},
	\mspace{10mu}
	M_{2}:=\begin{pmatrix}1 & 1 & 0 & 1\\0 & 1 & 0 & 0\\0 & 1 & 1 & 1\\0 & 2 & 0 & 1\end{pmatrix},
	\mspace{10mu}
	M_{3}:=\begin{pmatrix}1 & 0 & 1 & 1\\0 & 1 & 1 & 1\\0 & 0 & 1 & 0\\0 & 0 & 2 & 1\end{pmatrix}
\end{equation}
and set $M_{w}:=M_{w_{1}}\cdots M_{w_{m}}$ for $w=w_{1}\ldots w_{m}\in W_{*}$
($M_{\emptyset}:=\id{4\times 4}$).
Note that then for any $n\in\mathbb{N}\cup\{0\}$ we easily obtain
\begin{equation}\label{eq:M1nM2nM3n}\setlength\arraycolsep{3pt}
M_{1^{n}}=\begin{pmatrix}1 & 0 & 0 & 0\\n^{2} & 1 & 0 & n\\n^{2} & 0 & 1 & n\\2n & 0 & 0 & 1\end{pmatrix},
	\mspace{10mu}
	M_{2^{n}}=\begin{pmatrix}1 & n^{2} & 0 & n\\0 & 1 & 0 & 0\\0 & n^{2} & 1 & n\\0 & 2n & 0 & 1\end{pmatrix},
	\mspace{10mu}
	M_{3^{n}}=\begin{pmatrix}1 & 0 & n^{2} & n\\0 & 1 & n^{2} & n\\0 & 0 & 1 & 0\\0 & 0 & 2n & 1\end{pmatrix}.
\end{equation}
\end{dfn}
\begin{prop}\label{prop:curvatures-Mw}
Let $\mathscr{D}=(D_{1},D_{2},D_{3})\in\TDT$, let $\alpha,\beta,\gamma,\kappa$
be as in Proposition \textup{\ref{prop:circumscribed-inscribed}}, let $w\in W_{*}$
and $(D_{w,1},D_{w,2},D_{w,3}):=\mathscr{D}_{w}$. Then
\begin{equation}
\bigl(\curv(D_{w,1}),\curv(D_{w,2}),\curv(D_{w,3}),\kappa(\mathscr{D}_{w})\bigr)
	=(\alpha,\beta,\gamma,\kappa)M_{w}.
\label{eq:curvatures-Mw}
\end{equation}
\end{prop}
\begin{proof}
This follows by an induction in $|w|$ using
Proposition \ref{prop:circumscribed-inscribed}-(2) and Definition \ref{dfn:TDT-Phi}.
\end{proof}

We next collect basic facts regarding the Hausdorff dimension and measure of $K(\mathscr{D})$.
For each $s\in(0,+\infty)$ let $\mathscr{H}^{s}:2^{\mathbb{C}}\to[0,+\infty]$ denote
the $s$-dimensional Hausdorff (outer) measure on $\mathbb{C}$ with respect to the
Euclidean metric, and for each $A\subset\mathbb{C}$ let $\dim_{\mathrm{H}}A$ denote
its Hausdorff dimension; see, e.g., \cite[Chapters 4--7]{Mattila:GSMES} for details.
As is well known, it easily follows from the definition of $\mathscr{H}^{s}$
that the image $f(A)$ of $A\subset\mathbb{C}$ by
$f:A\to\mathbb{C}$ with $\lip_{A}f<+\infty$ satisfies
$\mathscr{H}^{s}(f(A))\leq(\lip_{A}f)^{s}\mathscr{H}^{s}(A)$ for any $s\in(0,+\infty)$
and hence in particular $\dim_{\mathrm{H}}f(A)\leq\dim_{\mathrm{H}}A$.
On the basis of this observation, we easily get the following lemma.
\begin{lem}\label{lem:AG-Hausdorff-dim}
Let $\mathscr{D},\mathscr{D}'\in\TDT$. Then there exists $c\in(0,+\infty)$ such that
$\mathscr{H}^{s}(K(\mathscr{D}))\leq c^{s}\mathscr{H}^{s}(K(\mathscr{D}'))$
for any $s\in(0,+\infty)$. In particular,
$\dim_{\mathrm{H}}K(\mathscr{D})=\dim_{\mathrm{H}}K(\mathscr{D}')$.
\end{lem}
\begin{proof}
Let $f_{\mathscr{D}',\mathscr{D}}$ denote the unique orientation-preserving
M\"{o}bius transformation on $\widehat{\mathbb{C}}$ such that
$f_{\mathscr{D}',\mathscr{D}}(q_{j}(\mathscr{D}'))=q_{j}(\mathscr{D})$ for any $j\in S$.
Then $f_{\mathscr{D}',\mathscr{D}}(K(\mathscr{D}'))=K(\mathscr{D})$, since a M\"{o}bius
transformation on $\widehat{\mathbb{C}}$ maps any open disk in $\widehat{\mathbb{C}}$
onto another. Now the assertion follows from the observation in the last paragraph and
$\lip_{\overline{D_{\mathrm{cir}}(\mathscr{D}')}}f_{\mathscr{D}',\mathscr{D}}<+\infty$.
\end{proof}
\begin{dfn}\label{dfn:AG-Hausdorff-dim}
Noting Lemma \ref{lem:AG-Hausdorff-dim}, we define
\begin{equation}\label{eq:AG-Hausdorff-dim}
d_{\mathsf{AG}}:=\dim_{\mathrm{H}}K(\mathscr{D}),
	\qquad\textrm{where $\mathscr{D}\in\TDT$ is arbitrary.}
\end{equation}
\end{dfn}
\begin{thm}[Boyd \cite{Boyd:Mathematika1973}; see also \cite{Hirst:JLMS1967,MauldinUrbanski:AdvMath1998,McMullen:AmerJMath1998}]\label{thm:AG-Hausdorff-dim-bounds}
\begin{equation}\label{eq:AG-Hausdorff-dim-bounds}
1.300197<d_{\mathsf{AG}}<1.314534.
\end{equation}
\end{thm}

Moreover, for the $d_{\mathsf{AG}}$-dimensional Hausdorff measure 
$\mathscr{H}^{d_{\mathsf{AG}}}(K(\mathscr{D}))$ of $K(\mathscr{D})$
we have the following theorem, which was proved first by Sullivan
\cite{Sullivan:Acta1984} through considerations on the isometric action of
M\"{o}bius transformations on the three-dimensional hyperbolic space, and 
later by Mauldin and Urba\'{n}ski \cite{MauldinUrbanski:AdvMath1998}
through purely two-dimensional arguments.
\begin{thm}[{\cite[Theorem 2]{Sullivan:Acta1984}, \cite[Theorem 2.6]{MauldinUrbanski:AdvMath1998}}]\label{thm:AG-Hausdorff-meas}
\begin{equation}\label{eq:AG-Hausdorff-meas}
0<\mathscr{H}^{d_{\mathsf{AG}}}(K(\mathscr{D}))<+\infty
	\qquad\textrm{for any $\mathscr{D}\in\TDT$.}
\end{equation}
\end{thm}
\begin{rmk}\label{rmk:AG-Hausdorff-meas}
The self-conformality of $K(\mathscr{D})$ is required most crucially in the proof of
Theorem \ref{thm:AG-Hausdorff-meas}, and is heavily used further to obtain certain
equicontinuity properties of $\{\mathscr{H}^{d_{\mathsf{AG}}}(K(\mathscr{D}_{w}))\}_{w\in W_{*}}$
as a family of functions of $\bigl(\curv(D_{1}),\curv(D_{2}),\curv(D_{3})\bigr)$, where
$(D_{1},D_{2},D_{3}):=\mathscr{D}$. This equicontinuity is the key to verifying the
ergodic-theoretic assumptions of Kesten's renewal theorem \cite[Theorem 2]{Kesten:AOP1974},
which is then applied to conclude Theorem \ref{thm:general-counting-AG} below.
\end{rmk}
%
\section{The canonical Dirichlet form on the Apollonian gasket}\label{sec:AG-DF}
In this section, we introduce the canonical Dirichlet form on the Apollonian gasket
$K(\mathscr{D})$, whose infinitesimal generator is our Laplacian on $K(\mathscr{D})$,
and state its properties established by the author in \cite{K:WeylAG}; see
\cite{FOT,CF} for the basics of the theory of regular symmetric Dirichlet forms.

Before giving its actual definition, we briefly summarize how it has been discovered.
The initial idea for its construction was suggested by the theory of analysis on the
\emph{harmonic Sierpi\'{n}ski gasket} $K_{\mathcal{H}}$ (Figure \ref{fig:SGs}, right)
due to Kigami \cite{Kig:hdSG,Kig:MRG}. This is a compact subset of $\mathbb{C}$
defined as the image of a \emph{harmonic map} $\Phi:K\to\mathbb{C}$ from
the \emph{Sierpi\'{n}ski gasket} $K$ (Figure \ref{fig:SGs}, left) to $\mathbb{C}$.
More precisely, let $V_{0}=\{q_{1},q_{2},q_{3}\}$ be the set of the three outmost
vertices of $K$, let $(\mathcal{E},\mathcal{F})$ be the (self-similar)
\emph{standard Dirichlet form} on $K$
(so that $\mathcal{F}$ is known to be a dense subalgebra of $(\mathcal{C}(K),\|\cdot\|_{\sup})$),
and let $h^{K}_{1},h^{K}_{2}\in\mathcal{F}$ be \emph{$\mathcal{E}$-harmonic} on $K\setminus V_{0}$
and satisfy $\mathcal{E}(h^{K}_{j},h^{K}_{k})=\delta_{jk}$ for any $j,k\in\{1,2\}$
(see \cite[Sections 2 and 3]{K:aghSG} and the references therein for details). Then we can
define a continuous map $\Phi:K\to\mathbb{C}$ by $\Phi(x):=\bigl(h^{K}_{1}(x),h^{K}_{2}(x)\bigr)$,
and its image $K_{\mathcal{H}}:=\Phi(K)$ is called the harmonic Sierpi\'{n}ski gasket.
\begin{figure}[t]\captionsetup{width=\linewidth}\centering
	\begin{minipage}{.480\linewidth}\centering
		\includegraphics[height=80pt]{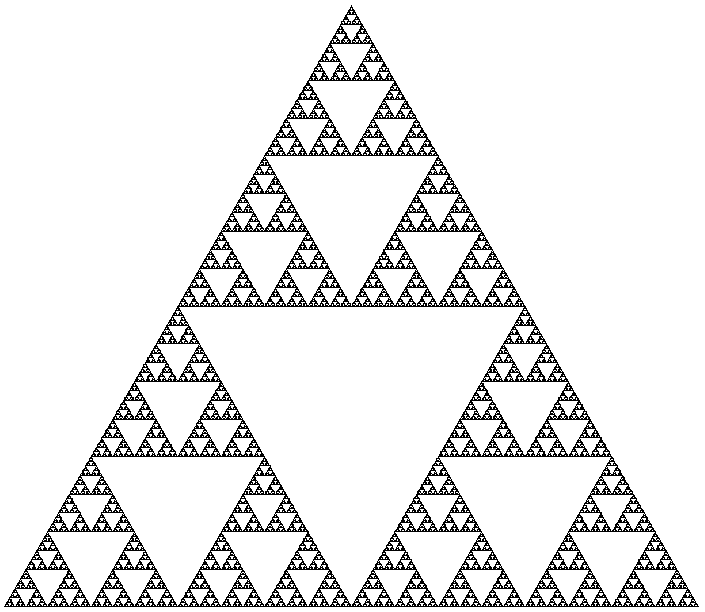}
	\end{minipage}\hspace*{6pt}
	\begin{picture}(0,0)
		\thicklines
		\put(-57,4){\vector(1,0){126}}
		\put(-56,17){$\Phi:=\biggl(\begin{matrix}h^{K}_{1}\\h^{K}_{2}\end{matrix}\biggr):K\to K_{\mathcal{H}}\hookrightarrow\mathbb{C}$}
		\put(-52,-8){$h^{K}_{j}$: $\mathcal{E}$-harmonic on $K\setminus V_{0}$}
		\put(-30,-22){$\mathcal{E}(h^{K}_{j},h^{K}_{k})=\delta_{jk}$}
	\end{picture}
	\begin{minipage}{.480\linewidth}\centering
		\includegraphics[height=80pt]{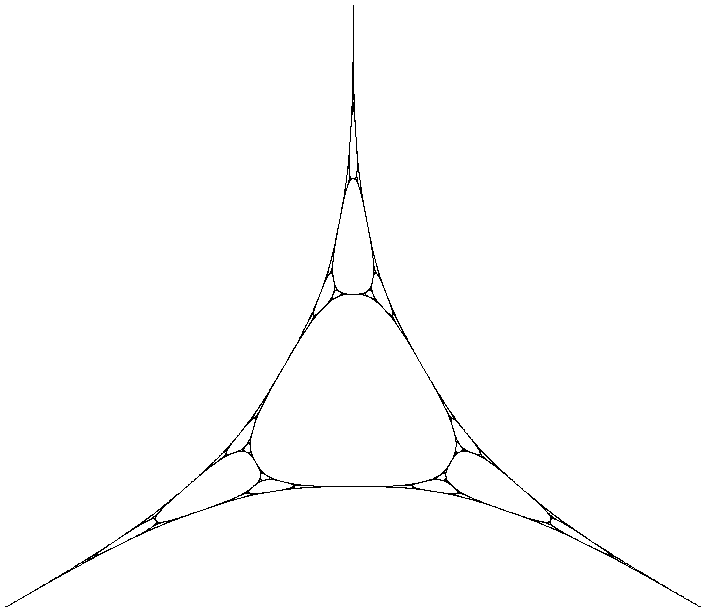}
	\end{minipage}
	\caption{Sierpi\'{n}ski gasket $K$ and harmonic Sierpi\'{n}ski gasket $K_{\mathcal{H}}$}\label{fig:SGs}
\end{figure}
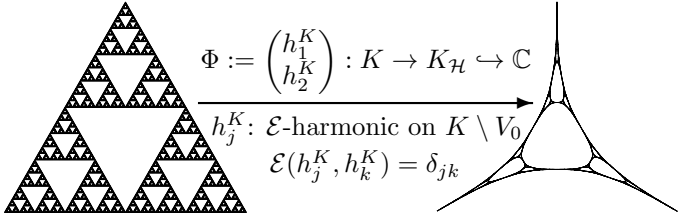
In fact, Kigami has proved in \cite[Theorem 3.6]{Kig:hdSG} that $\Phi:K\to K_{\mathcal{H}}$
is injective and hence a homeomorphism, and further in \cite[Theorem 4.1]{Kig:hdSG} that
a one-dimensional, measure-theoretic ``Riemannian structure'' can be defined on $K$
through the embedding $\Phi$ and the \emph{$\mathcal{E}$-energy measure} $\mu$%
\footnote{$\mu$ was first introduced in \cite{Kus:DF} and is called the \emph{Kusuoka measure} on $K$.}
of $\Phi$, which plays the role of the ``Riemannian volume measure'' and is given by
\begin{equation}\label{eq:Kusuoka-meas}
\mu:=\mu_{\langle h^{K}_{1}\rangle}+\mu_{\langle h^{K}_{2}\rangle}=\textrm{``}|\nabla\Phi|^{2}\,d\vol\textrm{''};
\end{equation}
here $\mu_{\langle u\rangle}$ denotes the $\mathcal{E}$-energy measure of $u\in\mathcal{F}$ playing
the role of ``$|\nabla u|^{2}\,d\vol$'' and defined as the unique Borel measure on $K$ such that
\begin{equation}\label{eq:energy-meas}
\int_{K}f\,d\mu_{\langle u\rangle}=\mathcal{E}(fu,u)-\frac{1}{2}\mathcal{E}(f,u^{2})
	\qquad\textrm{for any $f\in\mathcal{F}$.}
\end{equation}
Kigami has also proved in \cite[Theorem 6.3]{Kig:MRG} that the heat kernel of
$(K,\mu,\mathcal{E},\mathcal{F})$ satisfies the two-sided \emph{Gaussian} estimate
of the same form as for Riemannian manifolds, and further detailed studies of
$(K,\mu,\mathcal{E},\mathcal{F})$ have been done in \cite{K:stahSG,KoskelaZhou,K:aghSG};
see \cite{K:aghSG} and the references therein for details.

As observed from Figures \ref{fig:AGs} and \ref{fig:SGs}, the overall geometric structure
of the Apollonian gasket $K(\mathscr{D})$ resembles that of the harmonic Sierpi\'{n}ski
gasket $K_{\mathcal{H}}$, and then it is natural to expect that the above-mentioned
framework of the measurable Riemannian structure on $K$ induced by the embedding
$\Phi:K\to K_{\mathcal{H}}$ can be adapted to the setting of $K(\mathscr{D})$ for
$\mathscr{D}\in\TDTo$ to construct a ``geometrically canonical'' Dirichlet form on
$K(\mathscr{D})$. Namely, it is expected that there exists a non-zero strongly local
regular symmetric Dirichlet form $(\mathcal{E}^{\mathscr{D}},\mathcal{F}_{\mathscr{D}})$
over $K(\mathscr{D})$ \emph{with respect to which the coordinate functions
$\re(\cdot)|_{K(\mathscr{D})},\im(\cdot)|_{K(\mathscr{D})}$ are harmonic on
$K(\mathscr{D})\setminus V_{0}(\mathscr{D})$}.
The possibility of such a construction was first noted by Teplyaev in
\cite[Theorem 5.17]{Tep:energySG}, and in \cite{K:WeylAG} the author has completed
the construction of $(\mathcal{E}^{\mathscr{D}},\mathcal{F}_{\mathscr{D}})$ and
further proved its uniqueness and concrete identification, summarized as follows.
We start with some definitions.
\begin{dfn}\label{dfn:circular-arc}
\begin{itemize}[label=\textup{(1)},align=left,leftmargin=*,topsep=0pt,parsep=0pt,itemsep=0pt]
\item[\textup{(1)}]A subset $C$ of $\mathbb{C}$ is called a \emph{circular arc}
	if and only if $C=\{z_{0}+re^{i\theta}\mid\theta\in[\alpha,\beta]\}$
	for some $z_{0}\in\mathbb{C}$, $r\in(0,+\infty)$ and $\alpha,\beta\in\mathbb{R}$
	with $\alpha<\beta$. In this case we set $\cent(C):=z_{0}$, $\rad(C):=r$ and
	$D_{C}:=\interior\{(1-t)\cent(C)+tz\mid\textrm{$z\in C$, $t\in[0,1]$}\}$.
\item[\textup{(2)}]For a circular arc $C$, the length measure on $(C,\mathscr{B}(C))$ is
	denoted by $\mathscr{H}^{1}_{C}$, the gradient vector along $C$ at $x\in C$ of a function
	$u:C\to\mathbb{R}$ is denoted by $\nabla_{C}u(x)$ provided $u$ is differentiable at $x$, and we set
	$W^{1,2}(C):=\{u\in\mathbb{R}^{C}\mid\textrm{$u$ is a.c.\ on $C$, $|\nabla_{C}u|\in L^{2}(C,\mathscr{H}^{1}_{C})$}\}$,
	where ``a.c.''\ is an abbreviation of ``absolutely continuous''.
\item[\textup{(3)}]We define $h_{1},h_{2}:\mathbb{C}\to\mathbb{R}$ by
	$h_{1}(z):=\re z$ and $h_{2}(z):=\im z$.
\end{itemize}
\end{dfn}
\begin{dfn}\label{dfn:AG-arcs}
Let $\mathscr{D}=(D_{1},D_{2},D_{3})\in\TDTo$. We define
\begin{equation}\label{eq:AG-arcs}
\mathscr{A}_{\mathscr{D}}:=\{\overline{T(\mathscr{D})}\cap\partial D_{j}\mid j\in S\}
	\cup\{\partial D_{\mathrm{in}}(\mathscr{D}_{w})\mid w\in W_{*}\}
\end{equation}
and set $K^{0}(\mathscr{D}):=\bigcup_{C\in\mathscr{A}_{\mathscr{D}}}C$,
so that each $C\in\mathscr{A}_{\mathscr{D}}$ is a circular arc,
$\bigcup_{C\in\mathscr{A}_{\mathscr{D}}}D_{C}=\triangle(\mathscr{D})\setminus K(\mathscr{D})$,
$\bigcup_{C,A\in\mathscr{A}_{\mathscr{D}},\,C\not=A}(C\cap A)=\bigcup_{w\in W_{*}}V_{0}(\mathscr{D}_{w})$,
and an induction in $|w|$ easily shows that for any $w\in W_{*}$,
\begin{equation}\label{eq:AG-arcs-Kw}
\mathscr{A}_{\mathscr{D}_{w}}
	=\{C\cap K(\mathscr{D}_{w})\mid C\in\mathscr{A}_{\mathscr{D}}\}\setminus\{\emptyset\}.
\end{equation}
\end{dfn}
The canonical Dirichlet form $(\mathcal{E}^{\mathscr{D}},\mathcal{F}_{\mathscr{D}})$
on $K(\mathscr{D})$ and the associated ``Riemannian volume measure''
similar to \eqref{eq:Kusuoka-meas} turn out to be expressed explicitly
in terms of the circle packing structure of $K(\mathscr{D})$, as follows.
\begin{dfn}[cf.\ {\cite[Theorems 5.11 and 5.13]{K:WeylAG}}]\label{dfn:AG-volume-meas-identify}
Let $\mathscr{D}\in\TDTo$.
\begin{itemize}[label=\textup{(1)},align=left,leftmargin=*,topsep=0pt,parsep=0pt,itemsep=0pt]
\item[\textup{(1)}]We define a Borel measure $\mu^{\mathscr{D}}$ on $K(\mathscr{D})$ by
	\begin{equation}\label{eq:AG-volume-meas-identify}
	\mu^{\mathscr{D}}:=\sum\nolimits_{C\in\mathscr{A}_{\mathscr{D}}}\rad(C)\mathscr{H}^{1}_{C}(\cdot\cap C),
	\end{equation}
	so that for any $w\in W_{*}$ we have
	$\mu^{\mathscr{D}}(K(\mathscr{D}_{w}))=2\vol_{2}(\triangle(\mathscr{D}_{w}))$
	by \eqref{eq:AG-arcs-Kw},
	$\bigcup_{C\in\mathscr{A}_{\mathscr{D}_{w}}}D_{C}=\triangle(\mathscr{D}_{w})\setminus K(\mathscr{D}_{w})$
	and $\vol_{2}(K(\mathscr{D}_{w}))=0$.
\item[\textup{(2)}]
	For each $u\in\mathbb{R}^{K^{0}(\mathscr{D})}$ with $u|_{C}$ a.c.\ on $C$
	for any $C\in\mathscr{A}_{\mathscr{D}}$, we define a $\mu^{\mathscr{D}}$-a.e.\ defined,
	$\mathbb{R}^{2}$-valued Borel measurable map $\nabla_{\mathscr{D}}u$
	by $(\nabla_{\mathscr{D}}u)|_{C}:=\nabla_{C}(u|_{C})$
	for each $C\in\mathscr{A}_{\mathscr{D}}$, so that
	$|\nabla_{\mathscr{D}}u|^{2}\,d\mu^{\mathscr{D}}
		=\sum_{C\in\mathscr{A}_{\mathscr{D}}}|\nabla_{C}(u|_{C})|^{2}\rad(C)\,d\mathscr{H}^{1}_{C}$.
	Then we further define
	\begin{equation}\label{eq:AG-DF-identify-domain}
	\mathcal{F}_{\mathscr{D}}:=W^{1,2}_{\mathscr{D}}:=
		\biggl\{u\in\mathbb{R}^{K^{0}(\mathscr{D})}\biggm|
		\begin{minipage}{138.9pt}
			$u|_{C}\in W^{1,2}(C)$ for any $C\in\mathscr{A}_{\mathscr{D}}$,
			$|\nabla_{\mathscr{D}}u|\in L^{2}(K(\mathscr{D}),\mu^{\mathscr{D}})$
		\end{minipage}\biggr\}
	\end{equation}
	and set $\mathcal{C}_{\mathscr{D}}:=\{u\in\mathcal{C}(K(\mathscr{D}))\mid u|_{K^{0}(\mathscr{D})}\in\mathcal{F}_{\mathscr{D}}\}$
	and $\mathcal{C}^{\mathrm{lip}}_{\mathscr{D}}:=\{u\in\mathcal{C}(K(\mathscr{D}))\mid\lip_{K(\mathscr{D})}u<+\infty\}$,
	which are considered as linear subspaces of $\mathcal{F}_{\mathscr{D}}$ through the linear injection
	$\mathcal{C}(K(\mathscr{D}))\ni u\mapsto u|_{K^{0}(\mathscr{D})}\in\mathbb{R}^{K^{0}(\mathscr{D})}$.
	Noting that
	$\langle\nabla_{\mathscr{D}}u,\nabla_{\mathscr{D}}v\rangle\in L^{1}(K(\mathscr{D}),\mu^{\mathscr{D}})$
	for any $u,v\in\mathcal{F}_{\mathscr{D}}$, we also define a bilinear form
	$\mathcal{E}^{\mathscr{D}}:\mathcal{F}_{\mathscr{D}}\times\mathcal{F}_{\mathscr{D}}\to\mathbb{R}$
	on $\mathcal{F}_{\mathscr{D}}$ by
	\begin{equation}\label{eq:AG-DF-identify-form}
	\begin{split}
	\mathcal{E}^{\mathscr{D}}(u,v)
		:=&\int_{K(\mathscr{D})}\langle\nabla_{\mathscr{D}}u,\nabla_{\mathscr{D}}v\rangle\,d\mu^{\mathscr{D}}\\
	=&\sum\nolimits_{C\in\mathscr{A}_{\mathscr{D}}}\int_{C}\langle\nabla_{C}(u|_{C}),\nabla_{C}(v|_{C})\rangle\rad(C)\,d\mathscr{H}^{1}_{C}.
	\end{split}
	\end{equation}
	In particular, setting $d\mu^{\mathscr{D}}_{\langle u\rangle}:=|\nabla_{\mathscr{D}}u|^{2}\,d\mu^{\mathscr{D}}$
	for each $u\in\mathcal{F}_{\mathscr{D}}$, we have
	$\mu^{\mathscr{D}}=\mu^{\mathscr{D}}_{\langle h_{1}|_{K(\mathscr{D})}\rangle}+\mu^{\mathscr{D}}_{\langle h_{2}|_{K(\mathscr{D})}\rangle}$
	as the counterpart of \eqref{eq:Kusuoka-meas} for $K(\mathscr{D})$.
\end{itemize}
\end{dfn}
\begin{thm}[{\cite[Theorem 5.18]{K:WeylAG}}]\label{thm:AG-DF-regular}
Let $\mathscr{D}\in\TDTo$ and set
$\mathcal{F}_{\mathscr{D},0}^{0}:=\{u\in\mathcal{F}_{\mathscr{D}}\mid u|_{V_{0}(\mathscr{D})}=0\}$.
Then $(\mathcal{E}^{\mathscr{D}},\mathcal{F}_{\mathscr{D}})$ is an irreducible, strongly local,
regular symmetric Dirichlet form on $L^{2}(K(\mathscr{D}),\mu^{\mathscr{D}})$ with
a core $\mathcal{C}^{\mathrm{lip}}_{\mathscr{D}}$, and
\begin{equation}\label{eq:AG-spectral-gap}
\int_{K(\mathscr{D})}u^{2}\,d\mu^{\mathscr{D}}
	\leq 40\kappa(\mathscr{D})^{-2}\mathcal{E}^{\mathscr{D}}(u,u)
	\qquad\textrm{for any $u\in\mathcal{F}_{\mathscr{D},0}^{0}$.}
\end{equation}
Moreover, the inclusion map
$\mathcal{F}_{\mathscr{D}}\hookrightarrow L^{2}(K(\mathscr{D}),\mu^{\mathscr{D}})$
is a compact linear operator under the norm
$\|u\|_{\mathcal{F}_{\mathscr{D}}}:=(\mathcal{E}^{\mathscr{D}}(u,u)+\int_{K(\mathscr{D})}u^{2}\,d\mu^{\mathscr{D}})^{1/2}$
on $\mathcal{F}_{\mathscr{D}}$.
\end{thm}
\begin{thm}[{\cite[Theorem 5.23]{K:WeylAG}}]\label{thm:AG-DF-regular-unique}
Let $\mathscr{D}\in\TDTo$, let $\mu'$ be a finite Borel measure on $K(\mathscr{D})$
with $\mu'(U)>0$ for any non-empty open subset $U$ of $K(\mathscr{D})$, and let
$(\mathcal{E}',\mathcal{F}')$ be a strongly local, regular symmetric Dirichlet form on
$L^{2}(K(\mathscr{D}),\mu')$ with $\mathcal{E}'(u,u)>0$ for some $u\in\mathcal{F}'$.
Then the following two conditions are equivalent:
\begin{itemize}[label=\textup{(1)},align=left,leftmargin=*,topsep=0pt,parsep=0pt,itemsep=0pt]
\item[\textup{(1)}]Any $h\in\{h_{1}|_{K(\mathscr{D})},h_{2}|_{K(\mathscr{D})}\}$
	is in $\mathcal{F}'$ and is \emph{$\mathcal{E}'$-harmonic} on
	$K(\mathscr{D})\setminus V_{0}(\mathscr{D})$, i.e., $\mathcal{E}'(h,v)=0$ for any
	$v\in\mathcal{F}'\cap\mathcal{C}(K(\mathscr{D}))$ with $v|_{V_{0}(\mathscr{D})}=0$.
\item[\textup{(2)}]$\mathcal{F}'\cap\mathcal{C}(K(\mathscr{D}))=\mathcal{C}_{\mathscr{D}}$ and
	$\mathcal{E}'|_{\mathcal{C}_{\mathscr{D}}\times\mathcal{C}_{\mathscr{D}}}
		=c\mathcal{E}^{\mathscr{D}}|_{\mathcal{C}_{\mathscr{D}}\times\mathcal{C}_{\mathscr{D}}}$
	for some $c\in\mathbb{R}$.
\end{itemize}
\end{thm}
\begin{rmk}\label{rmk:AG-DF-regular-unique}
In contrast to the case of $K(\mathscr{D})$ described in Definition \ref{dfn:AG-volume-meas-identify},
Theorems \ref{thm:AG-DF-regular} and \ref{thm:AG-DF-regular-unique}, the standard
Dirichlet form $(\mathcal{E},\mathcal{F})$ on the Sierpi\'{n}ski gasket $K$ satisfies
$\mu_{\langle u\rangle}(K^{0})=0$ for any $u\in\mathcal{F}$ by \cite[Lemma 8.26]{K:aghSG}
and \cite[Lemma 5.7]{Hino:index}, where $K^{0}$ denotes the union of the boundaries of
the equilateral triangles constituting $K$. In particular, $(\mathcal{E},\mathcal{F})$
cannot be expressed as the sum of any weighted one-dimensional Dirichlet forms
on $\Phi(K^{0})\subset K_{\mathcal{H}}$ similar to \eqref{eq:AG-DF-identify-form}.
The author does not have a good explanation of the reason for this difference, and it
would be very nice to give one. A naive guess could be that some sufficient smoothness
of the relevant curves might be required for the validity of an expression like
\eqref{eq:AG-DF-identify-form} of a non-zero strongly local regular symmetric Dirichlet
form satisfying the analog of Theorem \ref{thm:AG-DF-regular-unique}-(1);
indeed, the curves constituting $\Phi(K^{0})$ are $\mathcal{C}^{1}$ but not $\mathcal{C}^{2}$
by \cite[Theorem 5.4-(2)]{Kig:MRG}, whereas the corresponding curves $C\in\mathscr{A}_{\mathscr{D}}$
in $K(\mathscr{D})$ are circular arcs and therefore real analytic. While this guess itself might
well be correct, it would be still unclear how smooth the relevant curves should need to be.
\end{rmk}
The rest of this section is devoted to a brief sketch of the proof of 
Theorems \ref{thm:AG-DF-regular} and \ref{thm:AG-DF-regular-unique}, which is rather
long and occupies the whole of \cite[Sections 4 and 5]{K:WeylAG}. It starts with
identifying what the \emph{trace} $\mathcal{E}^{\mathscr{D}}|_{V_{m}(\mathscr{D})}$,
\begin{equation}\label{eq:AG-DF-trace-Vm}
\mathcal{E}^{\mathscr{D}}|_{V_{m}(\mathscr{D})}(u,u)
	:=\inf_{v\in\mathcal{F}_{\mathscr{D}},\,v|_{V_{m}(\mathscr{D})}=u}\mathcal{E}^{\mathscr{D}}(v,v),
	\quad u\in\mathbb{R}^{V_{m}(\mathscr{D})},
\end{equation}
of $(\mathcal{E}^{\mathscr{D}},\mathcal{F}_{\mathscr{D}})$ to
$V_{m}(\mathscr{D}):=\bigcup_{w\in W_{m}}V_{0}(\mathscr{D}_{w})$
\emph{should} be for any $m\in\mathbb{N}\cup\{0\}$. In view of the desired
properties of $(\mathcal{E}^{\mathscr{D}},\mathcal{F}_{\mathscr{D}})$
in Theorem \ref{thm:AG-DF-regular-unique}, the forms
$\{\mathcal{E}^{\mathscr{D}}|_{V_{m}(\mathscr{D})}\}_{m\in\mathbb{N}\cup\{0\}}$
should have the properties in the following theorem.
\begin{thm}[{\cite[Theorem 5.17]{Tep:energySG}}]\label{thm:AG-Vm-Laplacians-exists}
Let $\mathscr{D}\in\TDTo$. Then there exists
$\{\mathcal{E}^{\mathscr{D}}_{m}\}_{m\in\mathbb{N}\cup\{0\}}$
such that the following hold for any $m\in\mathbb{N}\cup\{0\}$:
\begin{itemize}[label=\textup{(1)},align=left,leftmargin=*,topsep=0pt,parsep=0pt,itemsep=0pt]
\item[\textup{(1)}]$\mathcal{E}^{\mathscr{D}}_{m}$ is a symmetric Dirichlet form on $\ell^{2}(V_{m}(\mathscr{D}))$.
	$\mathcal{E}^{\mathscr{D}}_{m}(\ind{x},\ind{y})=0=\mathcal{E}^{\mathscr{D}}_{m}(\ind{x},\ind{})$
	for any $x,y\in V_{m}(\mathscr{D})$ with $\{\tau\in W_{m}\mid x,y\in V_{0}(\mathscr{D}_{\tau})\}=\emptyset$.%
\item[\textup{(2)}]Both $h_{1}|_{V_{m}(\mathscr{D})}$ and $h_{2}|_{V_{m}(\mathscr{D})}$ are
	$\mathcal{E}^{\mathscr{D}}_{m}$-harmonic on $V_{m}(\mathscr{D})\setminus V_{0}(\mathscr{D})$.
\item[\textup{(3)}]$\mathcal{E}^{\mathscr{D}}_{m}(u,u)
		=\min_{v\in\mathbb{R}^{V_{m+1}(\mathscr{D})},\,v|_{V_{m}(\mathscr{D})}=u}\mathcal{E}^{\mathscr{D}}_{m+1}(v,v)$
	for any $u\in\mathbb{R}^{V_{m}(\mathscr{D})}$.
\item[\textup{(4)}]$\mathcal{E}^{\mathscr{D}}_{m}(h_{1}|_{V_{m}(\mathscr{D})},h_{1}|_{V_{m}(\mathscr{D})})
		+\mathcal{E}^{\mathscr{D}}_{m}(h_{2}|_{V_{m}(\mathscr{D})},h_{2}|_{V_{m}(\mathscr{D})})
		=2\vol_{2}(\triangle(\mathscr{D}))$.
\end{itemize}
\end{thm}
Teplyaev's proof of Theorem \ref{thm:AG-Vm-Laplacians-exists} in \cite{Tep:energySG}
is purely Euclidean-geometric and provides no further information on
$\{\mathcal{E}^{\mathscr{D}}_{m}\}_{m\in\mathbb{N}\cup\{0\}}$.
The author has identified it as follows,
by applying a refinement of \cite[Corollary 4.2]{MeyersStrTep}.
\begin{thm}[{\cite[Theorem 4.18]{K:WeylAG}}]\label{thm:AG-Vm-Laplacians-unique}
For each $\mathscr{D}=(D_{1},D_{2},D_{3})\in\TDTo$, a sequence
$\{\mathcal{E}^{\mathscr{D}}_{m}\}_{m\in\mathbb{N}\cup\{0\}}$
as in Theorem \textup{\ref{thm:AG-Vm-Laplacians-exists}} is unique, and
\begin{equation}\label{eq:AG-V0-Laplacians}
\mathcal{E}^{\mathscr{D}}_{0}(u,u)
	=\sum_{j\in S}\frac{\kappa(\mathscr{D})^{2}+\curv(D_{j})^{2}}{2\kappa(\mathscr{D})\curv(D_{j})}\bigl(u(q_{j+1}(\mathscr{D}))-u(q_{j+2}(\mathscr{D}))\bigr)^{2}
\end{equation}
for any $u\in\mathbb{R}^{V_{0}(\mathscr{D})}$,
where $q_{j+3}(\mathscr{D}):=q_{j}(\mathscr{D})$ for $j\in S$.
Moreover, for any $\mathscr{D}\in\TDTo$, any $m\in\mathbb{N}\cup\{0\}$
and any $u\in\mathbb{R}^{V_{m}(\mathscr{D})}$,
\begin{equation}\label{eq:AG-Vm-Laplacians}
\mathcal{E}^{\mathscr{D}}_{m}(u,u)
	=\sum\nolimits_{w\in W_{m}}\mathcal{E}^{\mathscr{D}_{w}}_{0}(u|_{V_{0}(\mathscr{D}_{w})},u|_{V_{0}(\mathscr{D}_{w})}).
\end{equation}
\end{thm}
Let $\mathscr{D}\in\TDTo$. Theorem \ref{thm:AG-Vm-Laplacians-exists}-(3) allows us
to apply to $\{\mathcal{E}^{\mathscr{D}}_{m}\}_{m\in\mathbb{N}\cup\{0\}}$
the general theory from \cite[Chapter 2]{Kig:AOF} of constructing a Dirichlet form
by taking the \emph{``inductive limit''} of Dirichlet forms on finite sets. Namely,
setting $V_{*}(\mathscr{D}):=\bigcup_{m\in\mathbb{N}\cup\{0\}}V_{m}(\mathscr{D})$,
we can define a linear subspace $\mathcal{F}'_{\mathscr{D}}$ of
$\mathbb{R}^{V_{*}(\mathscr{D})}$ and a bilinear form
$\mathcal{E}'^{\mathscr{D}}:\mathcal{F}'_{\mathscr{D}}\times\mathcal{F}'_{\mathscr{D}}\to\mathbb{R}$
on $\mathcal{F}'_{\mathscr{D}}$ by
\begin{gather}\label{eq:DF-AG-domain}
\mathcal{F}'_{\mathscr{D}}:=\bigl\{u\in\mathbb{R}^{V_{*}(\mathscr{D})}\bigm|\lim\nolimits_{m\to\infty}\mathcal{E}^{\mathscr{D}}_{m}(u|_{V_{m}(\mathscr{D})},u|_{V_{m}(\mathscr{D})})<+\infty\bigr\},\\
\mathcal{E}'^{\mathscr{D}}(u,v):=\lim\nolimits_{m\to\infty}\mathcal{E}^{\mathscr{D}}_{m}(u|_{V_{m}(\mathscr{D})},v|_{V_{m}(\mathscr{D})})\in\mathbb{R},
	\quad u,v\in\mathcal{F}'_{\mathscr{D}}.
\label{eq:DF-AG-form}
\end{gather}

The next step of the proof of Theorems \ref{thm:AG-DF-regular} and \ref{thm:AG-DF-regular-unique}
is the following identification of $(\mathcal{E}'^{\mathscr{D}},\mathcal{F}'_{\mathscr{D}})$ as
$(\mathcal{E}^{\mathscr{D}},\mathcal{F}_{\mathscr{D}})$, i.e., as given by
\eqref{eq:AG-DF-identify-domain} and \eqref{eq:AG-DF-identify-form}.
\begin{thm}[{\cite[Theorem 5.13]{K:WeylAG}}]\label{thm:AG-DF-identify}
Let $\mathscr{D}\in\TDTo$. Then
$\mathcal{F}'_{\mathscr{D}}=\{u|_{V_{*}(\mathscr{D})}\mid u\in\mathcal{F}_{\mathscr{D}}\}$,
the mapping $\mathcal{F}_{\mathscr{D}}\ni u\mapsto u|_{V_{*}(\mathscr{D})}\in\mathcal{F}'_{\mathscr{D}}$
is a linear isomorphism, and
$\mathcal{E}'^{\mathscr{D}}(u|_{V_{*}(\mathscr{D})},v|_{V_{*}(\mathscr{D})})
	=\mathcal{E}^{\mathscr{D}}(u,v)$
for any $u,v\in\mathcal{F}_{\mathscr{D}}$.
\end{thm}
\begin{proof}[Sketch of the proof]
By Theorem \ref{thm:AG-Vm-Laplacians-exists}-(2),(3) and \eqref{eq:DF-AG-domain} we have
$h_{1}|_{V_{*}(\mathscr{D})},h_{2}|_{V_{*}(\mathscr{D})}\in\mathcal{F}'_{\mathscr{D}}$,
which together with \eqref{eq:DF-AG-domain} implies that
$\mathcal{C}'_{\mathscr{D}}:=\{u\in\mathcal{C}(K(\mathscr{D}))\mid u|_{V_{*}(\mathscr{D})}\in\mathcal{F}'_{\mathscr{D}}\}$
is a dense subalgebra of $(\mathcal{C}(K(\mathscr{D})),\|\mspace{-1.35mu}\cdot\mspace{-1.35mu}\|_{\sup})$ with
$h_{1}|_{K(\mathscr{D})},h_{2}|_{K(\mathscr{D})}\in\mathcal{C}^{\mathrm{lip}}_{\mathscr{D}}\subset\mathcal{C}'_{\mathscr{D}}$.
Hence at this stage we can already define the $\mathcal{E}'^{\mathscr{D}}$-energy measure
$\mu'^{\mathscr{D}}_{\langle u\rangle}$ of $u\in\mathcal{C}'_{\mathscr{D}}$ by
\eqref{eq:energy-meas} with $K(\mathscr{D}),\mathcal{E}'^{\mathscr{D}},\mathcal{C}'_{\mathscr{D}}$
in place of $K,\mathcal{E},\mathcal{F}$, and the analog of \eqref{eq:Kusuoka-meas} by
$\mu'^{\mathscr{D}}:=\mu'^{\mathscr{D}}_{\langle h_{1}|_{K(\mathscr{D})}\rangle}+\mu'^{\mathscr{D}}_{\langle h_{2}|_{K(\mathscr{D})}\rangle}$.
Then it follows from Theorem \ref{thm:AG-Vm-Laplacians-exists}-(4) and \eqref{eq:AG-Vm-Laplacians} that
$\mu'^{\mathscr{D}}(K(\mathscr{D}_{w}))=2\vol_{2}(\triangle(\mathscr{D}_{w}))=\mu^{\mathscr{D}}(K(\mathscr{D}_{w}))$
for any $w\in W_{*}$, whence $\mu'^{\mathscr{D}}=\mu^{\mathscr{D}}$.

Now that $\mu'^{\mathscr{D}}$ has been identified as $\mu^{\mathscr{D}}$ given by
\eqref{eq:AG-volume-meas-identify}, it is natural to guess%
\footnote{This is how the author first came up with the expressions
\eqref{eq:AG-DF-identify-domain} and \eqref{eq:AG-DF-identify-form}.}
that $\mathcal{F}'_{\mathscr{D}}\subset\{u|_{V_{*}(\mathscr{D})}\mid u\in\mathcal{F}_{\mathscr{D}}\}$ and that
$\mathcal{E}'^{\mathscr{D}}(u|_{V_{*}(\mathscr{D})},u|_{V_{*}(\mathscr{D})})=\mathcal{E}^{\mathscr{D}}(u,u)$
for any $u\in\mathcal{F}_{\mathscr{D}}$ with $u|_{V_{*}(\mathscr{D})}\in\mathcal{F}'_{\mathscr{D}}$.
This guess is not difficult to verify, first for any piecewise linear $u\in\mathcal{F}_{\mathscr{D}}$
by direct calculations based on Theorem \ref{thm:AG-Vm-Laplacians-exists}-(2), \eqref{eq:AG-V0-Laplacians},
\eqref{eq:AG-Vm-Laplacians} and \eqref{eq:DF-AG-form}, and then for any $u\in\mathcal{F}_{\mathscr{D}}$
with $u|_{V_{*}(\mathscr{D})}\in\mathcal{F}'_{\mathscr{D}}$ by using the canonical
approximation of $u$ by piecewise linear functions; here $u\in\mathcal{F}_{\mathscr{D}}$
is called \emph{$m$-piecewise linear}, where $m\in\mathbb{N}\cup\{0\}$,
if and only if $u|_{K^{0}(\mathscr{D}_{w})}$ is a linear combination of
$h_{1}|_{K^{0}(\mathscr{D}_{w})},h_{2}|_{K^{0}(\mathscr{D}_{w})},\ind{K^{0}(\mathscr{D}_{w})}$
for any $w\in W_{m}$, and \emph{piecewise linear} if and only if
$u$ is $m$-piecewise linear for some $m\in\mathbb{N}\cup\{0\}$.

Finally, for any $u\in\mathcal{F}_{\mathscr{D}}$, some direct calculations using
\eqref{eq:AG-V0-Laplacians}, \eqref{eq:AG-DF-identify-form} and \eqref{eq:AG-arcs-Kw} show that
$\mathcal{E}^{\mathscr{D}_{w}}_{0}(u|_{V_{0}(\mathscr{D}_{w})},u|_{V_{0}(\mathscr{D}_{w})})
	\leq 7\int_{K(\mathscr{D}_{w})}|\nabla_{\mathscr{D}}u|^{2}\,d\mu^{\mathscr{D}}$
for any $w\in W_{*}$, which together with \eqref{eq:AG-Vm-Laplacians} yields
$\mathcal{E}^{\mathscr{D}}_{m}(u|_{V_{m}(\mathscr{D})},u|_{V_{m}(\mathscr{D})})
	\leq\mathcal{E}^{\mathscr{D}}(u,u)$
for any $m\in\mathbb{N}\cup\{0\}$, whence
$u|_{V_{*}(\mathscr{D})}\in\mathcal{F}'_{\mathscr{D}}$ by \eqref{eq:DF-AG-domain}.
\end{proof}

The last main step of the proof of Theorem \ref{thm:AG-DF-regular} is to prove
\eqref{eq:AG-spectral-gap}, which is based mainly on \eqref{eq:AG-volume-meas-identify},
\eqref{eq:AG-DF-identify-form} and the following lemma.
\begin{lem}[{\cite[Lemma 5.19]{K:WeylAG}}]\label{lem:circular-arc-extension}
Let $C\subset\mathbb{C}$ be a circular arc, let $u\in\mathbb{R}^{C}$ satisfy $\lip_{C}u<+\infty$,
and for $a\in\mathbb{R}$ define $\mathcal{I}_{C}^{a}u:\overline{D_{C}}\to\mathbb{R}$ by
\begin{equation}\label{eq:circular-arc-extension}
\mathcal{I}_{C}^{a}u((1-t)\cent(C)+tz):=(1-t)a+tu(z),\quad(t,z)\in[0,1]\times C.
\end{equation}
Then for any $a\in[\min_{C}u,\max_{C}u]$,
$\lip_{\overline{D_{C}}}\mathcal{I}_{C}^{a}u\leq\sqrt{5}\lip_{C}u$ and
\begin{equation}\label{eq:circular-arc-W12}
\frac{2}{21}\int_{D_{C}}\mspace{-3mu}|\nabla\mathcal{I}_{C}^{a}u|^{2}\,d\vol_{2}
	\leq\int_{C}\mspace{-2mu}|\nabla_{C}u|^{2}\rad(C)\,d\mathscr{H}^{1}_{C}
	\leq 2\int_{D_{C}}\mspace{-3mu}|\nabla\mathcal{I}_{C}^{a}u|^{2}\,d\vol_{2}.
\end{equation}
Further, with $\overline{u}^{C}:=\mathscr{H}^{1}_{C}(C)^{-1}\int_{C}u\,d\mathscr{H}^{1}_{C}$,
for any $a\in\{0,\overline{u}^{C}\}$,
\begin{equation}\label{eq:circular-arc-L2}
2\int_{D_{C}}|\mathcal{I}_{C}^{a}u|^{2}\,d\vol_{2}
	\leq\int_{C}u^{2}\rad(C)\,d\mathscr{H}^{1}_{C}
	\leq 4\int_{D_{C}}|\mathcal{I}_{C}^{a}u|^{2}\,d\vol_{2}.
\end{equation}
\end{lem}
Combining Lemma \ref{lem:circular-arc-extension} with \eqref{eq:AG-volume-meas-identify}
and \eqref{eq:AG-DF-identify-form}, we obtain the following.
\begin{lem}[{\cite[Lemma 5.21]{K:WeylAG}}]\label{lem:AG-extension}
Let $\mathscr{D}\in\TDTo$ and $u\in\mathcal{C}^{\mathrm{lip}}_{\mathscr{D}}$. Noting
$\overline{\triangle(\mathscr{D})}\setminus(K(\mathscr{D})\setminus K^{0}(\mathscr{D}))
	=\bigcup_{C\in\mathscr{A}_{\mathscr{D}}}\overline{D_{C}}$,
define $\mathcal{I}_{\mathscr{D}}^{0}u\in\mathbb{R}^{\overline{\triangle(\mathscr{D})}}$ by
\begin{equation}\label{eq:AG-extension}
\mathcal{I}_{\mathscr{D}}^{0}u|_{K(\mathscr{D})}:=u,\quad
	\mathcal{I}_{\mathscr{D}}^{0}u|_{\overline{D_{C}}}:=
	\begin{cases}
		\mathcal{I}_{C}^{0}(u|_{C}) & \textrm{if $C\subset\partial T(\mathscr{D})$,}\\
		\mathcal{I}_{C}^{\overline{u}^{C}}(u|_{C}) & \textrm{if $C\not\subset\partial T(\mathscr{D})$,}
	\end{cases}
	\quad C\in\mathscr{A}_{\mathscr{D}}.
\end{equation}
If also $u|_{V_{0}(\mathscr{D})}=0$, then
$\mathcal{I}_{\mathscr{D}}^{0}u|_{\partial\triangle(\mathscr{D})}=0$,
$\lip_{\overline{\triangle(\mathscr{D})}}\mathcal{I}_{\mathscr{D}}^{0}u
	\leq\sqrt{5}\lip_{K(\mathscr{D})}u$,
\begin{gather}\label{eq:AG-extension-W12}
\frac{2}{21}\int_{\triangle(\mathscr{D})}|\nabla\mathcal{I}_{\mathscr{D}}^{0}u|^{2}\,d\vol_{2}
	\leq\mathcal{E}^{\mathscr{D}}(u,u)
	\leq 2\int_{\triangle(\mathscr{D})}|\nabla\mathcal{I}_{\mathscr{D}}^{0}u|^{2}\,d\vol_{2},\\
2\int_{\triangle(\mathscr{D})}|\mathcal{I}_{\mathscr{D}}^{0}u|^{2}\,d\vol_{2}
	\leq\int_{K(\mathscr{D})}u^{2}\,d\mu^{\mathscr{D}}
	\leq 4\int_{\triangle(\mathscr{D})}|\mathcal{I}_{\mathscr{D}}^{0}u|^{2}\,d\vol_{2}.
\label{eq:AG-extenstion-L2}
\end{gather}
\end{lem}
\begin{proof}[Sketch of the proof of Theorem \textup{\ref{thm:AG-DF-regular}}]
Recall the following classical fact implied by
\cite[Lemma 6.2.1, Theorems 4.5.1, 4.5.3 and 6.1.6]{Davies:STDO}: if $Q$ is an open
rectangle in $\mathbb{C}$ whose smaller side length is $\delta\in(0,+\infty)$, then
\begin{equation}\label{eq:spectral-gap-rectangle}
\int_{Q}u^{2}\,d\vol_{2}\leq\frac{\delta^{2}}{\pi^{2}}\int_{Q}|\nabla u|^{2}\,d\vol_{2}
\end{equation}
for any $u\in\mathbb{R}^{\overline{Q}}$ with $\lip_{\overline{Q}}u<+\infty$ and $u|_{\partial Q}=0$.
Since $\triangle(\mathscr{D})\subset Q$ for some such $Q$ with $\delta=3\kappa(\mathscr{D})^{-1}$
and then each $u\in\mathbb{R}^{\overline{\triangle(\mathscr{D})}}$ with
$\lip_{\overline{\triangle(\mathscr{D})}}u<+\infty$ and $u|_{\partial\triangle(\mathscr{D})}=0$
can be extended to $\overline{Q}$ by setting
$u|_{\overline{Q}\setminus\triangle(\mathscr{D})}:=0$ so as to satisfy
$\lip_{\overline{Q}}u<+\infty$ and $u|_{\partial Q}=0$, we easily see
from Lemma \ref{lem:AG-extension} and \eqref{eq:spectral-gap-rectangle} that
\eqref{eq:AG-spectral-gap} holds for any
$u\in\mathcal{F}_{\mathscr{D},0}^{0}\cap\mathcal{C}^{\mathrm{lip}}_{\mathscr{D}}$.

Now, by utilizing the canonical approximation of each $u\in\mathcal{F}_{\mathscr{D}}$ by
piecewise linear functions as in the sketch of the proof of Theorem \ref{thm:AG-DF-identify} above,
we can show that \eqref{eq:AG-spectral-gap} extends to any $u\in\mathcal{F}_{\mathscr{D},0}^{0}$,
which implies $\mathcal{F}_{\mathscr{D}}\subset L^{2}(K(\mathscr{D}),\mu^{\mathscr{D}})$, and that
the inclusion map $\mathcal{F}_{\mathscr{D}}\hookrightarrow L^{2}(K(\mathscr{D}),\mu^{\mathscr{D}})$
is the limit in operator norm of finite-rank linear operators and hence compact.
The rest of the proof is straightforward.
\end{proof}
\begin{proof}[Sketch of the proof of Theorem \textup{\ref{thm:AG-DF-regular-unique}}]
The implication from (2) to (1) is immediate from Theorem \ref{thm:AG-DF-identify} and
Theorem \ref{thm:AG-Vm-Laplacians-exists}-(2),(3). That from (1) to (2) can be shown by
defining the trace $\mathcal{E}|_{V_{m}(\mathscr{D})}$ of $(\mathcal{E},\mathcal{F})$
to $V_{m}(\mathscr{D})$ for $m\in\mathbb{N}\cup\{0\}$
in essentially the same way as \eqref{eq:AG-DF-trace-Vm}, proving that
$\{\mathcal{E}|_{V_{m}(\mathscr{D})}\}_{m\in\mathbb{N}\cup\{0\}}$ satisfies
Theorem \ref{thm:AG-Vm-Laplacians-exists}-(1),(2),(3) by the assumption of (1)
and then applying Theorem \ref{thm:AG-Vm-Laplacians-unique} to conclude that
$\{\mathcal{E}|_{V_{m}(\mathscr{D})}\}_{m\in\mathbb{N}\cup\{0\}}
	=\{c\mathcal{E}^{\mathscr{D}}_{m}\}_{m\in\mathbb{N}\cup\{0\}}$
for some $c\in\mathbb{R}$, which is easily seen to imply (2).
\end{proof}
%
\section{Weyl's eigenvalue asymptotics for the Apollonian gasket}\label{sec:Weyl-AG}
The following proposition is an easy consequence of Theorem \ref{thm:AG-DF-regular};
see also \cite[Exercise 4.2, Corollary 4.2.3, Theorems 4.5.1 and 4.5.3]{Davies:STDO}.
\begin{prop}\label{prop:Laplacian-eigenvalues}
Let $\mathscr{D}\in\TDTo$, let $V$ be a finite subset of $V_{*}(\mathscr{D})$ and set
$\mathcal{F}_{\mathscr{D},V}^{0}:=\{u\in\mathcal{F}_{\mathscr{D}}\mid u|_{V}=0\}$. Then
$(\mathcal{E}^{\mathscr{D}}|_{\mathcal{F}_{\mathscr{D},V}^{0}\times\mathcal{F}_{\mathscr{D},V}^{0}},\mathcal{F}_{\mathscr{D},V}^{0})$
is a strongly local, regular symmetric Dirichlet form on $L^{2}(K(\mathscr{D})\setminus V,\mu^{\mathscr{D}})$,
and there exists a unique non-decreasing sequence $\{\lambda^{\mathscr{D},V}_{n}\}_{n\in\mathbb{N}}\subset[0,+\infty)$
such that $-\mathcal{L}_{\mathscr{D},V}\varphi^{\mathscr{D},V}_{n}=\lambda^{\mathscr{D},V}_{n}\varphi^{\mathscr{D},V}_{n}$
for any $n\in\mathbb{N}$ for some complete orthonormal system
$\{\varphi^{\mathscr{D},V}_{n}\}_{n\in\mathbb{N}}\subset\mathcal{D}(\mathcal{L}_{\mathscr{D},V})$
of $L^{2}(K(\mathscr{D})\setminus V,\mu^{\mathscr{D}})$; here
$\mathcal{L}_{\mathscr{D},V}:\mathcal{D}(\mathcal{L}_{\mathscr{D},V})\to L^{2}(K(\mathscr{D})\setminus V,\mu^{\mathscr{D}})$
denotes the \emph{Laplacian}, i.e., the non-positive self-adjoint operator on
$L^{2}(K(\mathscr{D})\setminus V,\mu^{\mathscr{D}})$, associated with
$(\mathcal{E}^{\mathscr{D}}|_{\mathcal{F}_{\mathscr{D},V}^{0}\times\mathcal{F}_{\mathscr{D},V}^{0}},\mathcal{F}_{\mathscr{D},V}^{0})$.
Also, $\lim_{n\to\infty}\lambda^{\mathscr{D},V}_{n}=+\infty$, and for any $n\in\mathbb{N}$,
\begin{equation}\label{eq:min-max}
\lambda^{\mathscr{D},V}_{n}=\min\biggl\{\max_{u\in L\setminus\{0\}}\frac{\mathcal{E}^{\mathscr{D}}(u,u)}{\int_{K(\mathscr{D})}u^{2}\,d\mu^{\mathscr{D}}}\biggm|
	\begin{minipage}{90.5pt}
		$L$ is a linear subspace of $\mathcal{F}_{\mathscr{D},V}^{0}$, $\dim L=n$
	\end{minipage}\biggr\}.
\end{equation}
\end{prop}
The proof of the following theorem is the principal aim of \cite{K:WeylAG}.
\begin{thm}[{\cite[Theorem 7.1]{K:WeylAG}}]\label{thm:Weyl-AG}
There exists $c_{\mathsf{AG}}\in(0,+\infty)$ such that for any $\mathscr{D}\in\TDTo$
and any finite subset $V$ of $V_{*}(\mathscr{D})$,
\begin{equation}\label{eq:Weyl-AG}
\lim_{\lambda\to+\infty}\frac{\#\{n\in\mathbb{N}\mid\lambda^{\mathscr{D},V}_{n}\leq\lambda\}}{\lambda^{d_{\mathsf{AG}}/2}}
	=c_{\mathsf{AG}}\mathscr{H}^{d_{\mathsf{AG}}}(K(\mathscr{D})).
\end{equation}
\end{thm}
The rest of this section outlines the analytic aspects of the proof of
Theorem \ref{thm:Weyl-AG}. It can be deduced from the following theorem applicable
to more general counting functions, including the classical one given by
$\#\{w\in W_{*}\mid\curv(D_{\mathrm{in}}(\mathscr{D}_{w}))\leq\lambda\}$,
whose asymptotic behavior analogous to \eqref{eq:Weyl-AG} has been obtained
first by Oh and Shah in \cite[Corollary 1.8]{OhShah:InventMath2012}.
\begin{dfn}\label{dfn:AG-geometry-parametrize-index-I}
\begin{itemize}[label=\textup{(1)},align=left,leftmargin=*,topsep=0pt,parsep=0pt,itemsep=0pt]
\item[\textup{(1)}]We define $I:=\{j^{n}k\mid\textrm{$j,k\in S$, $j\not=k$, $n\in\mathbb{N}$}\}$,
	so that $I\subset W_{*}\setminus\{\emptyset\}$,
	$\tau\not\asymp\upsilon$ for any $\tau,\upsilon\in I$ with $\tau\not=\upsilon$ and
	\begin{equation}\label{eq:decomp-index-I}
	K(\mathscr{D})\setminus V_{0}(\mathscr{D})=\bigcup\nolimits_{\tau\in I}K(\mathscr{D}_{\tau})
		\qquad\textrm{for any $\mathscr{D}\in\TDT$.}
	\end{equation}
\item[\textup{(2)}]We define $\Gamma\subset[0,+\infty)^{4}$ by
	$\Gamma:=\{(g,\kappa(g))\mid\textrm{$g\in[0,+\infty)^{3}$, $\kappa(g)>0$}\}$,
	where $\kappa(g):=\sqrt{\beta\gamma+\gamma\alpha+\alpha\beta}$ for
	$g=(\alpha,\beta,\gamma)\in[0,+\infty)^{3}$, and set
	$\Gamma^{\circ}:=\Gamma\cap(0,+\infty)^{4}$, which is an open subset of $\Gamma$;
	recall Propositions \ref{prop:circumscribed-inscribed} and \ref{prop:curvatures-Mw}
	and note that $gM_{w}\in\Gamma$ for any $g\in\Gamma$ and any $w\in W_{*}$.
\item[\textup{(3)}]Recalling Theorem \ref{thm:AG-Hausdorff-meas}, we set
	$\HM(g):=\mathscr{H}^{d_{\mathsf{AG}}}(K(\mathscr{D}))$ for each $g=(\alpha,\beta,\gamma,\kappa)\in\Gamma$,
	where we take any $\mathscr{D}=(D_{1},D_{2},D_{3})\in\TDT$ with
	$\bigl(\curv(D_{1}),\curv(D_{2}),\curv(D_{3})\bigr)=(\alpha,\beta,\gamma)$,
	which is easily seen to exist. Note that $\HM(g)=s^{d_{\mathsf{AG}}}\HM(sg)$
	for any $(g,s)\in\Gamma\times(0,+\infty)$.
\end{itemize}
\end{dfn}
\begin{thm}[\cite{K:WeylAG}]\label{thm:general-counting-AG}
Let $\Gamma'$ denote either of $\Gamma$ and $\Gamma^{\circ}$, and
for each $n\in\mathbb{N}$ let $\lambda_{n}:\Gamma'\to(0,+\infty)$ be continuous and
satisfy $\lambda_{n}(sg)=s\lambda_{n}(g)$ for any $(g,s)\in\Gamma'\times(0,+\infty)$.
Suppose that $\lambda_{1}(g)=\min_{n\in\mathbb{N}}\lambda_{n}(g)$ and
$\lim_{n\to\infty}\lambda_{n}(g)=+\infty$ for any $g\in\Gamma'$, set
$\mathscr{N}(g,\lambda):=\#\{n\in\mathbb{N}\mid\lambda_{n}(g)\leq\lambda\}$
for $(g,\lambda)\in\Gamma'\times[0,+\infty)$, and suppose that
there exist $\eta\in[0,d_{\mathsf{AG}})$ and $c\in(0,+\infty)$ such that
for any $g=(\alpha,\beta,\gamma,\kappa)\in\Gamma'$ and any $\lambda\in(0,+\infty)$,
\begin{equation}\label{eq:general-counting-AG-remainder}
\begin{split}
&\sum\nolimits_{\tau\in I}\mathscr{N}(gM_{\tau},\lambda)\leq\mathscr{N}(g,\lambda)\\
&\mspace{30mu}\leq\sum\nolimits_{\tau\in I}\mathscr{N}(gM_{\tau},\lambda)
	+c(\min\{\beta+\gamma,\gamma+\alpha,\alpha+\beta\})^{-\eta}\lambda^{\eta}+c.
\end{split}
\end{equation}
Then there exists $c_{0}\in(0,+\infty)$ such that for any $g\in\Gamma'$,
\begin{equation}\label{eq:general-counting-AG}
\lim_{\lambda\to+\infty}\frac{\mathscr{N}(g,\lambda)}{\lambda^{d_{\mathsf{AG}}}}=c_{0}\HM(g).
\end{equation}
\end{thm}
Theorem \ref{thm:general-counting-AG} is proved by applying Kesten's renewal theorem
\cite[Theorem 2]{Kesten:AOP1974} to the Markov chain on $\widetilde{\Gamma}:=\{g\in\Gamma\mid\HM(g)=1\}$,
the \emph{``space of Euclidean shapes of $\{K(\mathscr{D})\}_{\mathscr{D}\in\TDT}$''},
with transition function $\mathscr{P}(g,\cdot):=\sum_{\tau\in I}\HM(gM_{\tau})\delta_{[gM_{\tau}]_{\Gamma}}$,
where for each $g\in\Gamma$ we set $[g]_{\Gamma}:=\HM(g)^{1/d_{\mathsf{AG}}}g\in\widetilde{\Gamma}$
and $\delta_{[g]_{\Gamma}}$ denotes the Borel probability measure on $\widetilde{\Gamma}$
with $\delta_{[g]_{\Gamma}}(\{[g]_{\Gamma}\})=1$; a brief sketch of the proof of
Theorem \ref{thm:general-counting-AG} can be found in \cite{K:WeylSurvErgodicTh},
and the full details will appear in \cite[Sections 3 and 7]{K:WeylAG}.
\begin{proof}[Sketch of the proof of Theorem \textup{\ref{thm:Weyl-AG}} under Theorem \textup{\ref{thm:general-counting-AG}}]
We define
$\mathscr{N}_{\mathscr{D},V}(\lambda):=\#\{n\in\mathbb{N}\mid\lambda^{\mathscr{D},V}_{n}\leq\lambda\}$,
$\mathscr{N}_{\mathscr{D}}(\lambda):=\mathscr{N}_{\mathscr{D},\emptyset}(\lambda)$ and
$\mathscr{N}_{\mathscr{D},0}(\lambda):=\mathscr{N}_{\mathscr{D},V_{0}(\mathscr{D})}(\lambda)$
for $\mathscr{D}\in\TDTo$, each finite subset $V$ of $V_{*}(\mathscr{D})$ and $\lambda\in[0,+\infty)$.
Then for any such $\mathscr{D},V,\lambda$, as noted in \cite[Theorem 4.1.7 and Corollary 4.1.8]{Kig:AOF},
we easily see from $\dim\mathcal{F}_{\mathscr{D}}/\mathcal{F}_{\mathscr{D},V}^{0}=\#V$ and \eqref{eq:min-max} that
$\lambda^{\mathscr{D},\emptyset}_{n}\leq\lambda^{\mathscr{D},V}_{n}\leq\lambda^{\mathscr{D},\emptyset}_{n+\#V}$
for any $n\in\mathbb{N}$ and thereby that
\begin{equation}\label{eq:eigenvalue-counting-DirichletBC-V}
\mathscr{N}_{\mathscr{D},V}(\lambda)\leq\mathscr{N}_{\mathscr{D}}(\lambda)
	\leq\mathscr{N}_{\mathscr{D},V}(\lambda)+\#V,
\end{equation}
so that it suffices to prove \eqref{eq:Weyl-AG} for $V=V_{0}(\mathscr{D})$, i.e.,
for $\mathscr{N}_{\mathscr{D},0}(\lambda)$.

To apply Theorem \ref{thm:general-counting-AG}, for each $n\in\mathbb{N}$ and each
$g=(\alpha,\beta,\gamma,\kappa)\in\Gamma^{\circ}$ we set
$\lambda_{n}(g):=(\lambda^{\mathscr{D},V_{0}(\mathscr{D})}_{n})^{1/2}$,
where we take any $\mathscr{D}=(D_{1},D_{2},D_{3})\in\TDTo$ with
$\bigl(\curv(D_{1}),\curv(D_{2}),\curv(D_{3})\bigr)=(\alpha,\beta,\gamma)$,
so that $\lambda_{n}(g)\geq\lambda_{1}(g)>0$ by
$\{u\in\mathcal{F}_{\mathscr{D}}\mid\mathcal{E}^{\mathscr{D}}(u,u)=0\}=\mathbb{R}\ind{}$
and $\lim_{n\to\infty}\lambda_{n}(g)=+\infty$ by Proposition \ref{prop:Laplacian-eigenvalues}.
We also easily see from Proposition \ref{prop:curvatures-Mw}, \eqref{eq:AG-volume-meas-identify},
\eqref{eq:AG-DF-identify-form} and \eqref{eq:min-max} that for any $n\in\mathbb{N}$,
$\lambda_{n}:\Gamma^{\circ}\to(0,+\infty)$ is continuous and satisfies
$\lambda_{n}(sg)=s\lambda_{n}(g)$ for any $(g,s)\in\Gamma^{\circ}\times(0,+\infty)$.

It remains to verify that $\{\lambda_{n}\}_{n\in\mathbb{N}}$ satisfies
\eqref{eq:general-counting-AG-remainder}. To this end, let
$\mathscr{D}=(D_{1},D_{2},D_{3})\in\TDTo$, $(\alpha,\beta,\gamma,\kappa)=:g$ be as in
Proposition \ref{prop:circumscribed-inscribed} and $\lambda\in(0,+\infty)$. Then since
$\#\{n\in\mathbb{N}\mid\lambda_{n}(gM_{w})\leq\lambda^{1/2}\}=\mathscr{N}_{\mathscr{D}_{w},0}(\lambda)$
for any $w\in W_{*}$ by Proposition \ref{prop:curvatures-Mw},
\eqref{eq:general-counting-AG-remainder} for $\{\lambda_{n}\}_{n\in\mathbb{N}}$
can be rephrased as
\begin{equation}\label{eq:eigenvalue-counting-AG-remainder}
\begin{split}
&\sum\nolimits_{\tau\in I}\mathscr{N}_{\mathscr{D}_{\tau},0}(\lambda)\leq\mathscr{N}_{\mathscr{D},0}(\lambda)\\
&\mspace{27mu}\leq\sum\nolimits_{\tau\in I}\mathscr{N}_{\mathscr{D}_{\tau},0}(\lambda)+c(\min\{\beta+\gamma,\gamma+\alpha,\alpha+\beta\})^{-\eta}\lambda^{\eta/2}+c,
\end{split}
\end{equation}
which can be shown with $\eta=1<d_{\mathsf{AG}}$ (recall Theorem \ref{thm:AG-Hausdorff-dim-bounds})
as follows. Set $c_{g}:=\min\{\beta+\gamma,\gamma+\alpha,\alpha+\beta\}$ and
$n_{\lambda}:=\min\{n\in\mathbb{N}\mid c_{g}^{2}n^{2}\geq 40\lambda\}$.
Then for any $n\in\mathbb{N}$ with $n\geq n_{\lambda}$, any $j\in S$ and
any $\tau\in I\cup\{j^{n_{\lambda}}\}$ with $\tau\leq j^{n_{\lambda}}$,
from \eqref{eq:AG-spectral-gap}, \eqref{eq:min-max} and \eqref{eq:M1nM2nM3n} we obtain
\begin{equation}\label{eq:eigenvalue-counting-AG-end-zero}
\mspace{-3mu}
\mathscr{N}_{\mathscr{D}_{\tau},0}(\lambda)=0\mspace{8mu}\textrm{by}\mspace{8mu}
\lambda^{\mathscr{D}_{\tau},V_{0}(\mathscr{D}_{\tau})}_{1}
	\geq\frac{\kappa(\mathscr{D}_{\tau})^{2}}{40}
	\geq\frac{\kappa(\mathscr{D}_{j^{n_{\lambda}}})^{2}}{40}
	>\frac{c_{g}^{2}n_{\lambda}^{2}}{40}\geq\lambda.\mspace{-5mu}
\end{equation}
On the other hand, setting $I_{\lambda}:=\{\tau\in I\mid|\tau|\leq n_{\lambda}\}\cup\{j^{n_{\lambda}}\mid j\in S\}$ and
$V_{\lambda}:=\bigcup_{\tau\in I_{\lambda}}V_{0}(\mathscr{D}_{\tau})$, we have
$K(\mathscr{D})\setminus V_{\lambda}
	=\bigcup_{\tau\in I_{\lambda}}(K(\mathscr{D}_{\tau})\setminus V_{0}(\mathscr{D}_{\tau}))$
with the union disjoint, which together with \eqref{eq:min-max} and
\eqref{eq:eigenvalue-counting-AG-end-zero} easily implies that
\begin{equation}\label{eq:eigenvalue-counting-AG-decomp-lambda}
\mathscr{N}_{\mathscr{D},V_{\lambda}}(\lambda)
	=\sum\nolimits_{\tau\in I_{\lambda}}\mathscr{N}_{\mathscr{D}_{\tau},0}(\lambda)
	=\sum\nolimits_{\tau\in I}\mathscr{N}_{\mathscr{D}_{\tau},0}(\lambda).
\end{equation}
Now \hspace*{-0.1pt}\eqref{eq:eigenvalue-counting-AG-remainder}\hspace*{-0.1pt} follows from
\hspace*{-0.1pt}\eqref{eq:eigenvalue-counting-AG-decomp-lambda}, $\#V_{\lambda}=9n_{\lambda}-3$ and the fact that
$\mathscr{N}_{\mathscr{D},V_{\lambda}}(\lambda)\leq\mathscr{N}_{\mathscr{D},0}(\lambda)
	\leq\mathscr{N}_{\mathscr{D},V_{\lambda}}(\lambda)+\#V_{\lambda}-3$
by the same proof as \eqref{eq:eigenvalue-counting-DirichletBC-V}.
Theorem \ref{thm:general-counting-AG} is thus applicable to $\{\lambda_{n}\}_{n\in\mathbb{N}}$
and yields \eqref{eq:general-counting-AG}, which means \eqref{eq:Weyl-AG}.
\end{proof}
%
\section{Kleinian groups with limit sets round Sierpi\'{n}ski carpets}\label{sec:WeylRSC}
In this last section, we illustrate the possibility of extending the results
in \S\ref{sec:AG-DF} and \S\ref{sec:Weyl-AG} to other circle packing fractals,
by presenting the results of the author's recent study in \cite{K:WeylRSC} obtained
as the initial step toward developing a rich theory of construction and analysis of
``geometrically canonical'' Laplacians on more general self-conformal fractals.

Let $\mob(\widehat{\mathbb{C}})$ denote the group of (orientation preserving or reversing)
M\"{o}bius transformations on $\widehat{\mathbb{C}}$. A discrete subgroup $G$ of
$\mob(\widehat{\mathbb{C}})$ is called a \emph{Kleinian group}%
\footnote{Kleinian groups are usually assumed to consist only of orientation preserving
elements, but here we allow them to contain orientation reversing ones.},
and the smallest closed subset $\partial_{\infty}G$ of $\widehat{\mathbb{C}}$
invariant with respect to the action of $G$ is called the \emph{limit set} of $G$.
It is known in the theory of Kleinian groups
(see, e.g., \cite{BullettMantica,KeenMaskitSeries,KeenSeries,Wright}) that the limit sets of certain classes
of Kleinian groups are circle packing fractals, and typical examples of such circle packing fractals
are provided in the book \cite{MSW:indra} together with a number of beautiful pictures of them.%
\begin{figure}[t]\centering\hspace*{-.046\linewidth}%
	\begin{minipage}{.680\linewidth}\captionsetup{width=\linewidth}\centering
		\includegraphics[height=115pt]{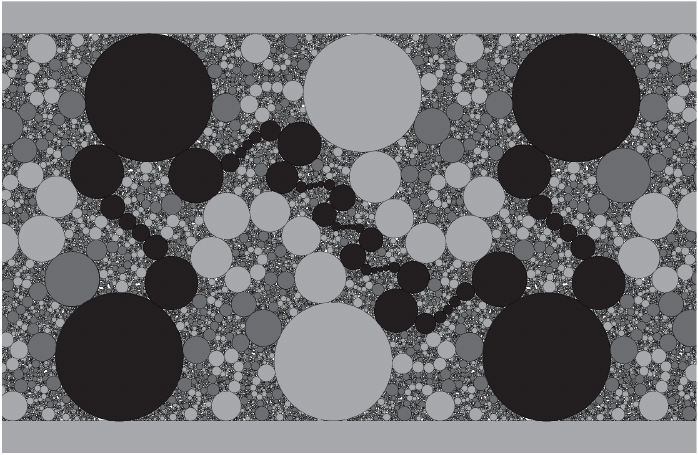}
		\caption{Limit set of $\frac{7}{43}$ double cusp group}\label{fig:7-43cusp}
	\end{minipage}\hspace*{-.070\linewidth}%
	\begin{minipage}{.460\linewidth}\captionsetup{width=\linewidth}\centering
		\includegraphics[height=115pt]{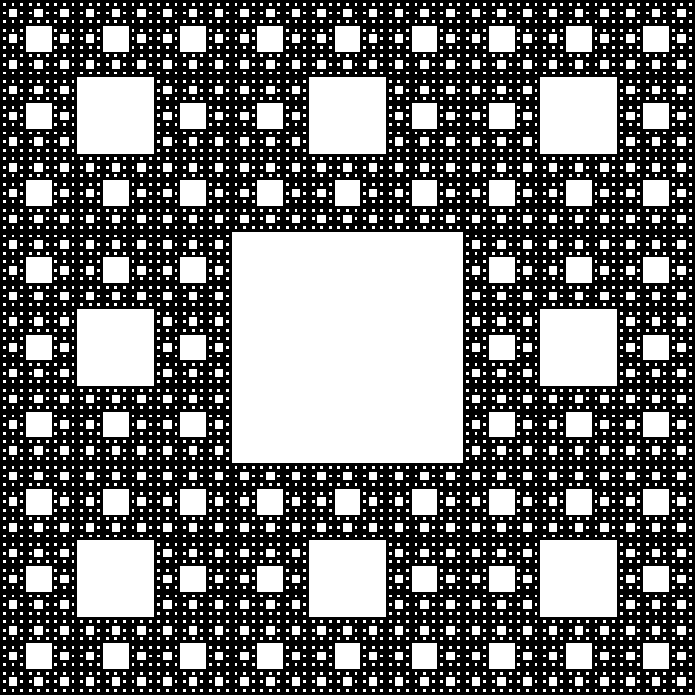}
		\caption{Sierpi\'{n}ski carpet}\label{fig:SC}
	\end{minipage}\hspace*{-.034\linewidth}%
\end{figure}

Since the expressions \eqref{eq:AG-volume-meas-identify} of $\mu^{\mathscr{D}}$
and \eqref{eq:AG-DF-identify-form} of the unique canonical Dirichlet form
$(\mathcal{E}^{\mathscr{D}},\mathcal{F}_{\mathscr{D}})$ on $K(\mathscr{D})$ makes
sense on a general circle packing fractal, (a candidate of) a ``geometrically canonical''
Laplacian on it can be defined by \eqref{eq:AG-volume-meas-identify} and
\eqref{eq:AG-DF-identify-form}, and it is natural to expect Weyl's eigenvalue
asymptotics to hold when the fractal has some nice self-conformal structure.
The author has recently verified this expectation in \cite{K:WeylDCGMaskit,K:WeylRSC} 
for the circle packing fractals arising as the limit sets of two specific classes of
Kleinian groups, one of which studied in \cite{K:WeylDCGMaskit} is the
\emph{double cusp groups} on the boundary of Maskit's embedding of the Teichm\"{u}ller
space of the once-punctured torus treated in detail in \cite{KeenSeries,MSW:indra,Wright}.
In this case, the limit sets (Figure \ref{fig:7-43cusp}) can be shown to admit a self-conformal
cellular decomposition similar to \eqref{eq:decomp-index-I} which is \emph{finitely ramified}
in the sense that any cell intersects the others only on boundedly many points, and
this property makes the proof of Weyl's asymptotics largely analogous to that of
Theorem \ref{thm:Weyl-AG}; a brief presentation of the precise statements of the results
can be found in \cite{K:MFO2016}, and the full details will be given in \cite{K:WeylDCGMaskit}.

On the other hand, each Kleinian group in the other class, which has been studied
in \cite{K:WeylRSC}, has as its limit set a \emph{round Sierpi\'{n}ski carpet}
(Figure \ref{fig:RSCs}), i.e., a subset of $\widehat{\mathbb{C}}$ homeomorphic to
the standard Sierpi\'{n}ski carpet (Figure \ref{fig:SC}) whose complement in
$\widehat{\mathbb{C}}$ consists of disjoint open disks in $\widehat{\mathbb{C}}$.
In particular, this limit set is \emph{infinitely ramified}, i.e., is not finitely
ramified regardless of the choice of a cellular decomposition, which prevents
the method of the above proof of \eqref{eq:eigenvalue-counting-AG-remainder} from
applying to it and thereby makes the proof of Weyl's asymptotics
for this case considerably more difficult.%
\begin{figure}[t]\captionsetup{width=\linewidth}\centering
\subcaptionbox{$\thirdangle=8$\label{fig:RSC8}}[.333\linewidth]{%
	\includegraphics[height=100pt]{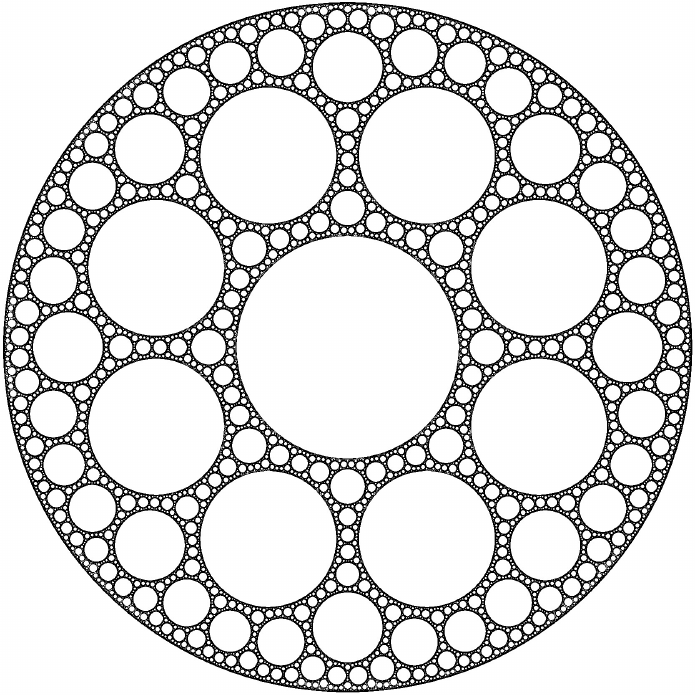}}%
\subcaptionbox{$\thirdangle=9$\label{fig:RSC9}}[.333\linewidth]{%
	\includegraphics[height=100pt]{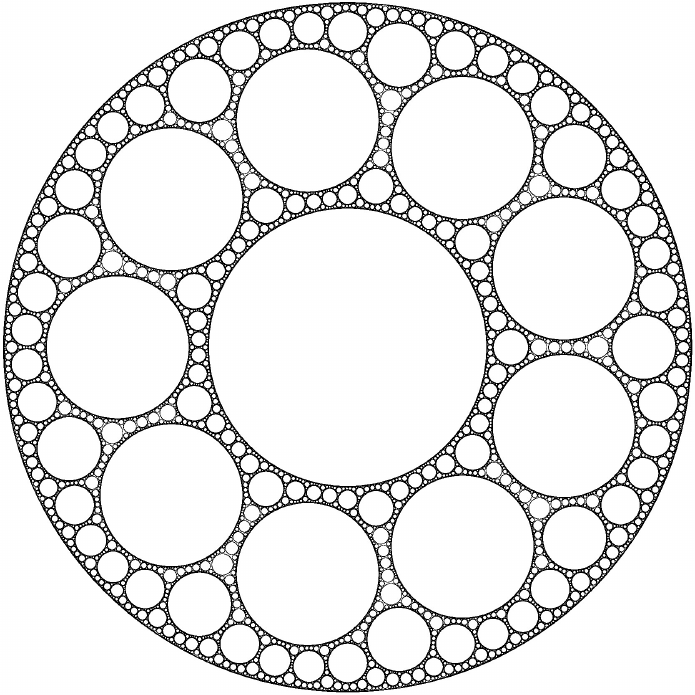}}%
\subcaptionbox{$\thirdangle=12$\label{fig:RSC12}}[.333\linewidth]{%
	\includegraphics[height=100pt]{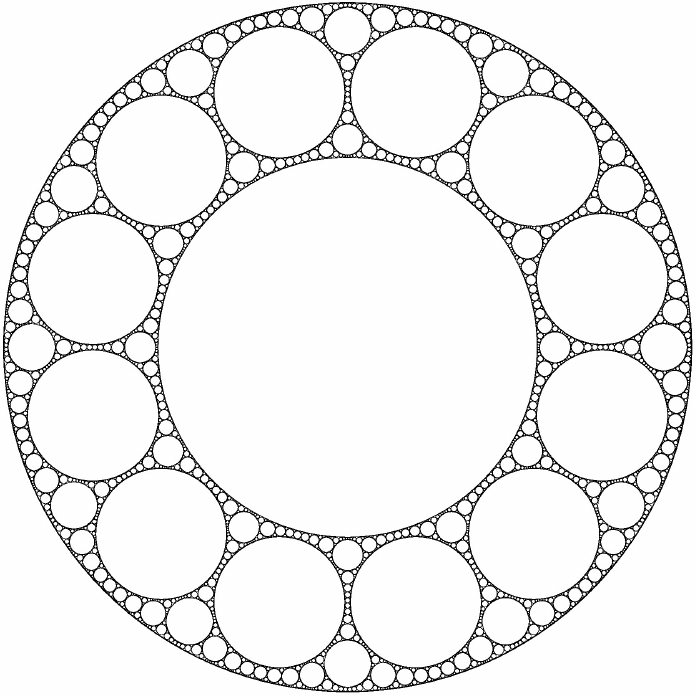}}%
\caption{The limit sets $\partial_{\infty}G_{\thirdangle}$ of the Kleinian groups $G_{\thirdangle}$}
\label{fig:RSCs}
\end{figure}

The rest of this section is devoted to a brief summary of the results in \cite{K:WeylRSC}
for the latter class of Kleinian groups, which are defined as follows.
Let $\thirdangle\in\mathbb{N}$ satisfy $\thirdangle>6$. It is a well-known fact from
hyperbolic geometry (see, e.g., \cite[Theorem 3.5.6]{Ratcliffe:Hyp3rd}) that by
$\frac{\pi}{2}+\frac{\pi}{3}+\frac{\pi}{\thirdangle}<\pi$ there exists a geodesic
triangle with inner angles $\frac{\pi}{2},\frac{\pi}{3},\frac{\pi}{\thirdangle}$,
unique up to hyperbolic isometry, in the Poincar\'{e} disk model
$\mathbb{D}:=\{z\in\mathbb{C}\mid|z|<1\}$ of the hyperbolic plane;
here we make the following specific choice of such one. The following construction
is a slight modification of that given in \cite{BullettMantica}.
\begin{dfn}\label{dfn:Gammaq-Gq}
\begin{itemize}[label=\textup{(1)},align=left,leftmargin=*,topsep=0pt,parsep=0pt,itemsep=0pt]
\item[\textup{(1)}]Set $\ell_{1}:=\mathbb{R}$,
	$\ell_{3}:=\{te^{i\pi/\thirdangle}\mid t\in\mathbb{R}\}$ and choose
	$t_{\thirdangle},s_{\thirdangle}\in(0,+\infty)$ so that
	$\ell_{2}:=\{z\in\mathbb{C}\mid|z-t_{\thirdangle}e^{i\pi/\thirdangle}|=s_{\thirdangle}\}$
	is orthogonal to $\partial\mathbb{D}$ and intersects $\ell_{1}$ with angle $\frac{\pi}{3}$;
	there is a unique such choice of $t_{\thirdangle},s_{\thirdangle}$
	by virtue of $\frac{\pi}{2}+\frac{\pi}{3}+\frac{\pi}{\thirdangle}<\pi$.
	The closed geodesic triangle in $\mathbb{D}$ formed by $\ell_{1},\ell_{2},\ell_{3}$
	is denoted by $\triangle_{\thirdangle}$, and the subgroup of $\mob(\widehat{\mathbb{C}})$
	generated by $\{\inv_{\ell_{k}}\}_{k=1}^{3}$ is denoted by $\Gamma_{\thirdangle}$,
	where $\inv_{\ell}$ denotes the inversion (reflection) in a circle
	or a straight line $\ell\subset\mathbb{C}$.
\item[\textup{(2)}]Choose $r_{\thirdangle}\in(0,1)$ so that
	$\ell_{4}:=\{z\in\mathbb{C}\mid|z|=r_{\thirdangle}\}$ intersects $\ell_{2}$
	with angle $\frac{\pi}{3}$; it is easy to see that there is a unique such choice
	of $r_{\thirdangle}$. The subgroup of $\mob(\widehat{\mathbb{C}})$ generated by
	$\{\inv_{\ell_{k}}\}_{k=1}^{4}$ is denoted by $G_{\thirdangle}$.
\end{itemize}
\end{dfn}
%
\begin{figure}[t]\captionsetup{width=\linewidth}\centering
\subcaptionbox{Inversion circles $\{\ell_{k}\}_{k}$\label{fig:RSC8-InvCircS}}[.333\linewidth]{%
	\includegraphics[height=100pt]{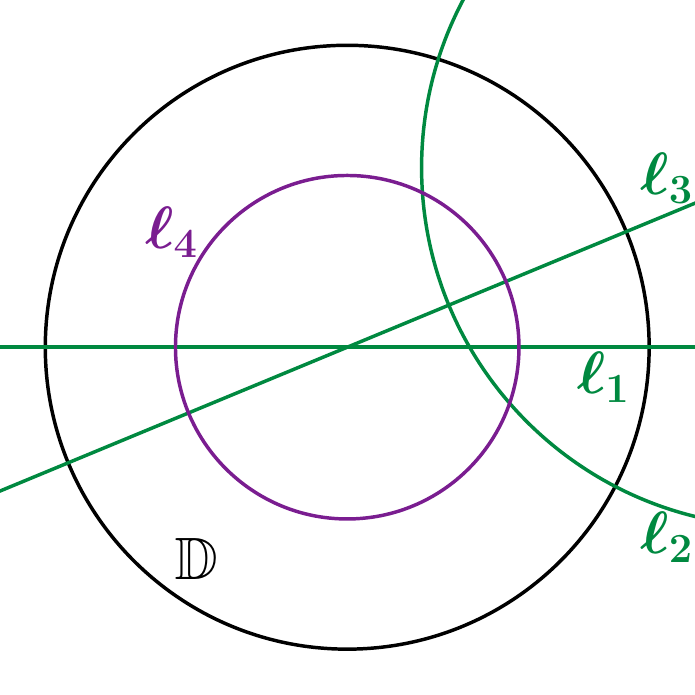}}%
\subcaptionbox{Tessellation by $\Gamma_{8}$\label{fig:RSC8-tessellationS}}[.333\linewidth]{%
	\includegraphics[height=100pt]{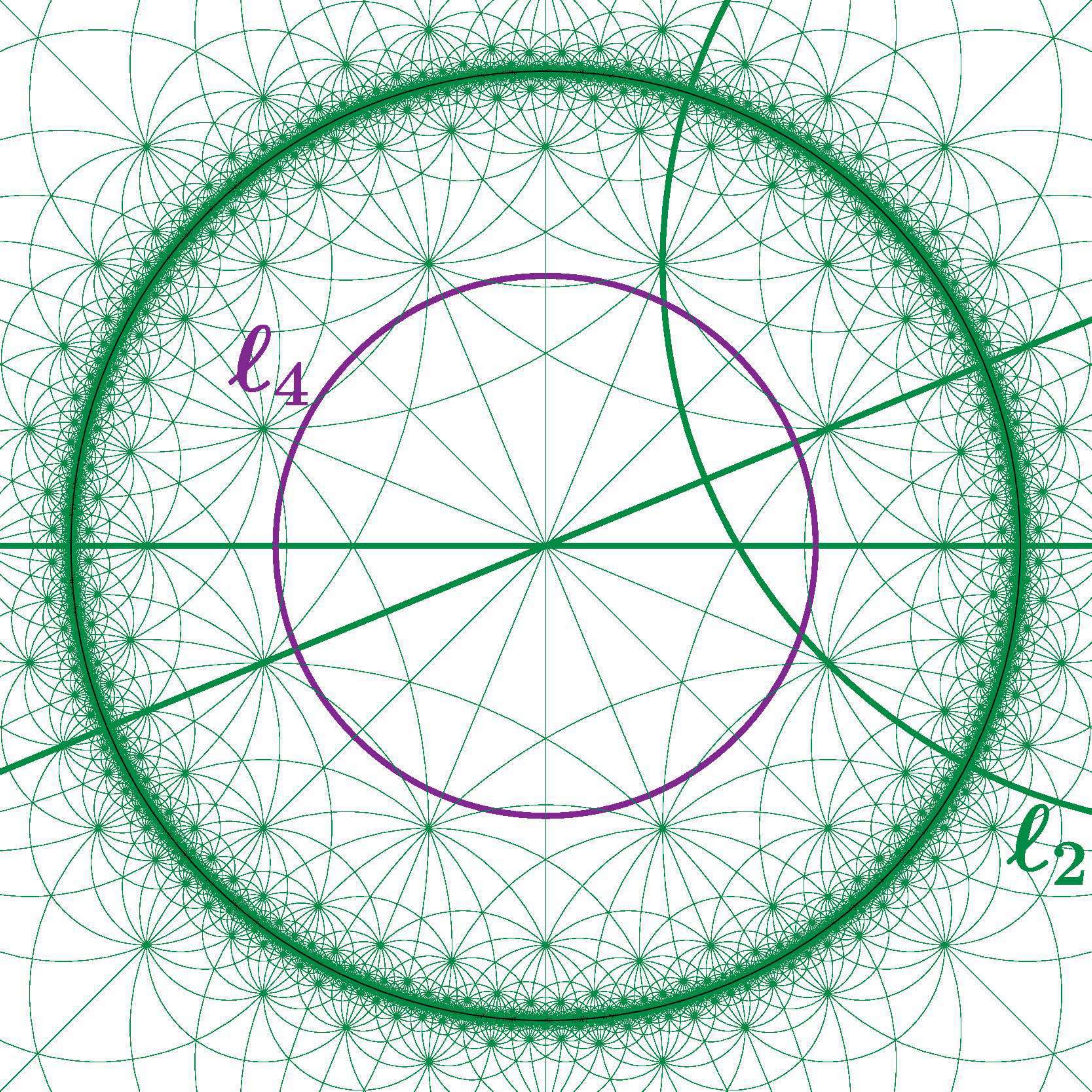}}%
\subcaptionbox{Construction of $\partial_{\infty}G_{8}$\label{fig:RSC8-1stLevel}}[.333\linewidth]{%
	\includegraphics[height=100pt]{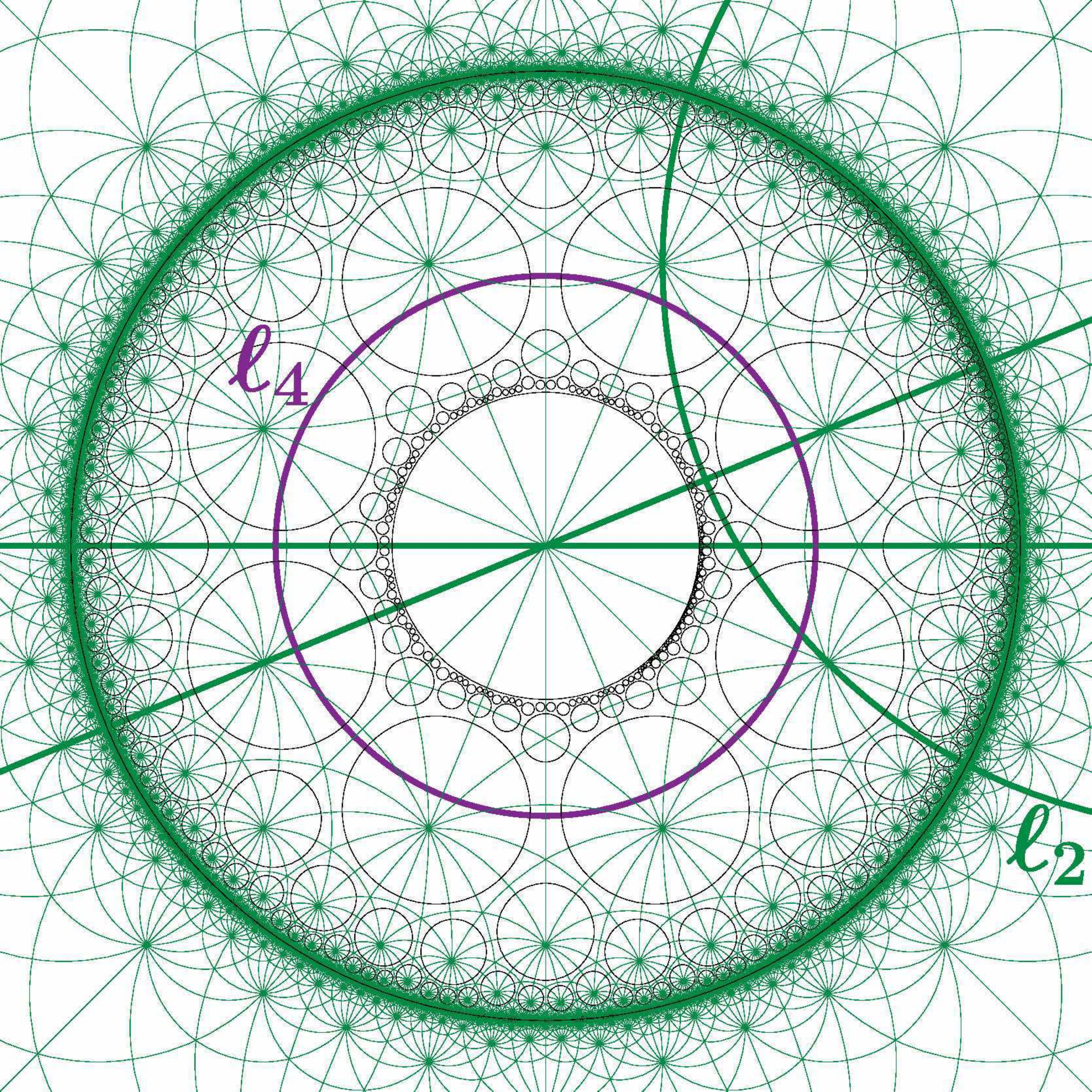}}%
\caption{Illustration of Definition \ref{dfn:Gammaq-Gq} and Proposition \ref{prop:Gammaq-Gq}: $\Gamma_{8},\partial_{\infty}G_{8}$}
\label{fig:RSC8-construct}
\end{figure}
%
\begin{prop}\label{prop:Gammaq-Gq}
\begin{itemize}[label=\textup{(1)},align=left,leftmargin=*,topsep=0pt,parsep=0pt,itemsep=0pt]
\item[\textup{(1)}]$\mathbb{D}=\bigcup_{\tau\in\Gamma_{\thirdangle}}\tau(\triangle_{\thirdangle})$
	and $\tau(\interior\triangle_{\thirdangle})\cap\upsilon(\triangle_{\thirdangle})=\emptyset$
	for any $\tau,\upsilon\in\Gamma_{\thirdangle}$ with $\tau\not=\upsilon$.
\item[\textup{(2)}]$G_{\thirdangle}$ is a Kleinian group,
	$\partial_{\infty}G_{\thirdangle}=\overline{\bigcup_{\tau\in G_{\thirdangle}}\tau(\partial\mathbb{D})}
		=\widehat{\mathbb{C}}\setminus\bigcup_{\tau\in G_{\thirdangle}}\tau(\widehat{\mathbb{C}}\setminus\overline{\mathbb{D}})$,
	$D_{1}\cap D_{2}=\emptyset$ for any
	$D_{1},D_{2}\in\bigl\{\tau(\widehat{\mathbb{C}}\setminus\mathbb{D})\bigm|\tau\in G_{\thirdangle}\bigr\}$
	with $D_{1}\not=D_{2}$, and $\interior\partial_{\infty}G_{\thirdangle}=\emptyset$.
	In particular, $\partial_{\infty}G_{\thirdangle}$ is a round Sierpi\'{n}ski carpet.
\end{itemize}
\end{prop}
\begin{proof}
(1) is immediate from \emph{Poincar\'{e}'s polygon theorem}
(see, e.g., \cite[Theorem 7.1.3]{Ratcliffe:Hyp3rd}), which applies to
$\triangle_{\thirdangle}$ since any of its inner angles is a \emph{submultiple}
of $\pi$, i.e., of the form $\pi/n$ for some $n\in\mathbb{N}\cup\{+\infty\}$.

For (2), recall (see, e.g., \cite[Sections 4.4--4.6]{Ratcliffe:Hyp3rd}) that
$\mob(\widehat{\mathbb{C}})$ is canonically isomorphic to the group
of isometries of the upper half-space model $\mathbb{H}^{3}:=\mathbb{C}\times(0,+\infty)$
of the three-dimensional hyperbolic space, where the inversion $\inv_{\ell}$ in
a circle or a straight line $\ell\subset\mathbb{C}$ corresponds to the inversion in the
sphere or the plane $\widetilde{\ell}$ intersecting $\mathbb{C}$ orthogonally on $\ell$.
Then since the closed polyhedron $\triangle_{\thirdangle}^{3}$ in $\mathbb{H}^{3}$
formed by $\{\widetilde{\ell_{k}}\}_{k=1}^{4}$, defined as the part of
$\{re^{i\theta}\mid(r,\theta)\in[0,+\infty)\times[0,\frac{\pi}{q}]\}\times(0,+\infty)$
above $\widetilde{\ell_{2}}$ and $\widetilde{\ell_{4}}$, has only submultiples of $\pi$ as
the dihedral angles between its faces, by \emph{Poincar\'{e}'s polyhedron theorem}
(see, e.g., \cite[Theorem 13.5.2]{Ratcliffe:Hyp3rd}) applied to $\triangle_{\thirdangle}^{3}$
we have $\mathbb{H}^{3}=\bigcup_{\tau\in G_{\thirdangle}}\tau(\triangle_{\thirdangle}^{3})$ and
$\tau(\interior_{\mathbb{H}^{3}}\triangle_{\thirdangle}^{3})\cap\upsilon(\triangle_{\thirdangle}^{3})=\emptyset$
for any $\tau,\upsilon\in G_{\thirdangle}$ with $\tau\not=\upsilon$.
Now we can obtain the first three assertions from this fact,
$\interior\partial_{\infty}G_{\thirdangle}=\emptyset$ from \cite[Theorem 12.2.7]{Ratcliffe:Hyp3rd},
and the last one from the topological characterization of the
Sierpi\'{n}ski carpet in \cite{Whyburn:SC}.
\end{proof}

Even though in Definition \ref{dfn:Gammaq-Gq} we have specifically chosen the unit disk
$\mathbb{D}$ and the geodesic triangle $\triangle_{\thirdangle}$, a particular choice
of a disk $D$ in $\mathbb{C}$ and a geodesic triangle in $D$ should not matter for the
desired Laplacian eigenvalue asymptotics. We should note also that the expressions
\eqref{eq:AG-volume-meas-identify} and \eqref{eq:AG-DF-identify-form} do not make
perfect sense for the family $\{\tau(\partial\mathbb{D})\mid\tau\in G_{\thirdangle}\}$ of
circles constituting $\partial_{\infty}G_{\thirdangle}$, since $\partial\mathbb{D}$ should
be treated together with the part $\widehat{\mathbb{C}}\setminus\overline{\mathbb{D}}$ of
$\widehat{\mathbb{C}}\setminus\partial_{\infty}G_{\thirdangle}$ enclosed by $\partial\mathbb{D}$
and thereby considered to be of \emph{infinite} area and radius, which is incompatible
with \eqref{eq:AG-volume-meas-identify} and \eqref{eq:AG-DF-identify-form}.
To take care of these issues, we introduce the following definition.
\begin{dfn}\label{dfn:RSC-interior-shapes}
We define
$\mathcal{G}:=\bigl\{g\in\mob(\widehat{\mathbb{C}})\bigm|g^{-1}(\infty)\in\widehat{\mathbb{C}}\setminus\overline{\mathbb{D}}\bigr\}$,
and for each $g\in\mathcal{G}$ we set
$\mathscr{D}_{g}:=\bigl\{g\tau(\widehat{\mathbb{C}}\setminus\overline{\mathbb{D}})\bigm|\tau\in G_{\thirdangle}\bigr\}\setminus\bigl\{g(\widehat{\mathbb{C}}\setminus\overline{\mathbb{D}})\bigr\}$
and $K_{g}:=g(\mathbb{D}\cap\partial_{\infty}G_{\thirdangle})=g(\mathbb{D})\setminus\bigcup_{D\in\mathscr{D}_{g}}D$,
so that $\mathscr{D}_{g}$ is a family of open disks in $\mathbb{C}$ and
$\overline{D_{1}}\subset g(\mathbb{D})\setminus\overline{D_{2}}$
for any $D_{1},D_{2}\in\mathscr{D}_{g}$ with $D_{1}\not=D_{2}$.
\end{dfn}
\begin{dfn}[\cite{K:WeylRSC}]\label{dfn:Dirichlet-form-RSC}
Let $g\in\mathcal{G}$.
We define a linear subspace $\mathcal{C}_{g}$ of $\mathcal{C}_{\mathrm{c}}(K_{g})$ by
$\mathcal{C}_{g}:=\{u\in\mathcal{C}_{\mathrm{c}}(K_{g})\mid\lip_{K_{g}}u<+\infty\}$,
and also define a \emph{finite} Borel measure $\mu^{g}$ on $K_{g}$ and a bilinear form
$\mathcal{E}^{g}:\mathcal{C}_{g}\times\mathcal{C}_{g}\to\mathbb{R}$ on $\mathcal{C}_{g}$ by
\begin{gather}\label{eq:measure-RSC}
\mu^{g}:=\sum\nolimits_{D\in\mathscr{D}_{g}}\rad(D)\mathscr{H}^{1}_{\partial D}(\cdot\cap\partial D),\\
\mathcal{E}^{g}(u,v):=\sum\nolimits_{D\in\mathscr{D}_{g}}
	\int_{\partial D}\langle\nabla_{\partial D}(u|_{\partial D}),\nabla_{\partial D}(v|_{\partial D})\rangle\rad(D)\,d\mathscr{H}^{1}_{\partial D}.
\label{eq:form-RSC}
\end{gather}
\end{dfn}
\begin{prop}[\cite{K:WeylRSC}]\label{prop:Dirichlet-form-RSC}
Let $g\in\mathcal{G}$. Then $(\mathcal{E}^{g},\mathcal{C}_{g})$ is closable in $L^{2}(K_{g},\mu^{g})$
and its smallest closed extension $(\mathcal{E}^{g},\mathcal{F}_{g})$ in $L^{2}(K_{g},\mu^{g})$
is a strongly local, regular symmetric Dirichlet form on $L^{2}(K_{g},\mu^{g})$.
Further, the inclusion map $\mathcal{F}_{g}\hookrightarrow L^{2}(K_{g},\mu^{g})$
is a compact linear operator under the norm
$\|u\|_{\mathcal{F}_{g}}:=(\mathcal{E}^{g}(u,u)+\int_{K_{g}}u^{2}\,d\mu^{g})^{1/2}$
on $\mathcal{F}_{g}$.
\end{prop}
\begin{prop}[\cite{K:WeylRSC}]\label{prop:harmonic-RSC}
Let $g\in\mathcal{G}$. Then any $h\in\{h_{1}|_{K_{g}},h_{2}|_{K_{g}}\}$ is
$\mathcal{E}^{g}$-harmonic on $K_{g}$, i.e., $\mathcal{E}^{g}(h,v)=0$ for any $v\in\mathcal{C}_{g}$,
with $\mathcal{E}^{g}(h,v)$ still defined by \eqref{eq:form-RSC}.
\end{prop}
\begin{proof}
This follows easily by explicit calculations using the Gauss--Green theorem
and the fact that $\partial D$ is a circle for any $D\in\mathscr{D}_{g}$.
\end{proof}

The following is the main result of \cite{K:WeylRSC}.
Note that for any $g\in\mathcal{G}$ and any non-empty open subset $U$ of $K_{g}$,
$d_{\thirdangle}:=\dim_{\mathrm{H}}\partial_{\infty}G_{\thirdangle}=\dim_{\mathrm{H}}K_{g}\in(1,2)$
and $\mathscr{H}^{d_{\thirdangle}}(U)\in(0,+\infty)$ by \cite[Theorem 7]{Sullivan:IHESPublMath79}
and $\lip_{\overline{\mathbb{D}}}g<+\infty$, and Proposition \ref{prop:Dirichlet-form-RSC} implies the
analog of Proposition \ref{prop:Laplacian-eigenvalues} for $(\mathcal{E}^{g,U},\mathcal{F}_{g,U}^{0})$
on $L^{2}(U,\mu^{g}|_{U})$, where $\mu^{g}|_{U}\mspace{-1.15mu}:=\mspace{-1.15mu}\mu^{g}|_{\mathscr{B}(U)}$,
$\mathcal{F}_{g,U}^{0}\mspace{-1.15mu}:=\mspace{-1.15mu}\overline{\{u\in\mathcal{C}_{g}\mid\supp_{K_{g}}[u]\subset U\}}^{\mathcal{F}_{g}}$
and $\mathcal{E}^{g,U}:=\mathcal{E}^{g}|_{\mathcal{F}_{g,U}^{0}\times\mathcal{F}_{g,U}^{0}}$.
\begin{thm}[\cite{K:WeylRSC}]\label{thm:WeylRSC}
There exists $c_{\thirdangle}\in(0,+\infty)$ such that for any $g\in\mathcal{G}$ and any
non-empty open subset $U$ of $K_{g}$ with $\mathscr{H}^{d_{\thirdangle}}(\partial_{K_{g}}U)=0$
and $\overline{U}\subset g(\mathbb{D})$, the eigenvalues $\{\lambda^{g,U}_{n}\}_{n\in\mathbb{N}}$
(repeated according to multiplicity) of the Laplacian on $L^{2}(U,\mu^{g}|_{U})$
associated with $(\mathcal{E}^{g,U},\mathcal{F}_{g,U}^{0})$ satisfy
\begin{equation}\label{eq:WeylRSC}
\lim_{\lambda\to+\infty}\frac{\#\{n\in\mathbb{N}\mid\lambda^{g,U}_{n}\leq\lambda\}}{\lambda^{d_{\thirdangle}/2}}
	=c_{\thirdangle}\mathscr{H}^{d_{\thirdangle}}(U).
\end{equation}
\end{thm}
The ergodic-theoretic aspects of the proof of Theorem \ref{thm:WeylRSC} are largely analogous
to those of the proof of Theorem \ref{thm:Weyl-AG}, and in particular the roles played by the
self-conformality of $K_{g}$ are similar to those described in Remark \ref{rmk:AG-Hausdorff-meas}.
The most difficult part of the proof of Theorem \ref{thm:WeylRSC} is that of an analog of
\eqref{eq:eigenvalue-counting-AG-remainder}, which is achieved by heavy use of heat kernel estimates
in combination with the property of $\{\tau(\partial\mathbb{D})\mid\tau\in G_{\thirdangle}\}$
that they are \emph{uniformly relatively separated} in the following sense (see \cite{Bonk:InventMath2011}):
\begin{equation}\label{eq:uniform-relative-separation-RSC}
\inf_{(x,y)\in C_{1}\times C_{2}}|x-y|\geq\varepsilon_{q}\min\{\rad(C_{1}),\rad(C_{2})\}
\end{equation}
for any $C_{1},C_{2}\in\{\tau(\partial\mathbb{D})\mid\tau\in G_{\thirdangle}\}$ with
$C_{1}\not=C_{2}$ for some $\varepsilon_{q}\in(0,+\infty)$. The full details of the
proof of Theorem \ref{thm:WeylRSC} will appear in \cite{K:WeylRSC}.

%
%

%
\end{document}